\documentclass[12pt,reqno]{amsart}
\usepackage{amssymb}
\usepackage{amsmath}
\usepackage{amsfonts}
\usepackage{amsthm}
\usepackage{mathrsfs}
\usepackage{enumitem}
\usepackage{extarrows}
\usepackage{xcolor}
\usepackage{graphicx, float, afterpage}
\usepackage{bm}
\usepackage{tikz}
\usepackage{amsmath}
\usepackage{mathtools}
\usepackage{caption}
\usepackage{subcaption}
\usepackage{hyperref}
\usepackage{diagbox}
\usepackage[noadjust]{cite}
\usepackage[margin=1in]{geometry}
\newtheorem{thm}{Theorem}[section]

\newtheorem{pro}[thm]{Proposition}
\newtheorem{cor}[thm]{Corollary}
\newtheorem{exam}[thm]{Example}

\newtheorem{lemma}[thm]{Lemma}
\newtheorem{que}[thm]{Question}

\numberwithin{equation}{section}

\usepackage{cellspace}
\usepackage[normalem]{ulem}
\newcommand\redout{\bgroup\markoverwith{\textcolor{red}{\rule[0.5ex]{4pt}{0.8pt}}}\ULon}

\newcommand\stout[1]{\ifmmode\text{\redout{\ensuremath{#1}}}\else\redout{#1}\fi}

\begin{document}
\title{Spectral properties of a class of Sierpinski-type Moran measures on \( \mathbb{R}^n \)}
\author{Jia-Long Chen}
\author{Xiao-Yu Yan$^{*}$}
\address{ School of Mathematics, South China University of Technology, Guangzhou 510641, P.R. China.}
\address{Key Laboratory of Computing and Stochastic Mathematics (Ministry of Education), School of Mathematics and Statistics, Hunan Normal University, Changsha 410081, P.R. China.}
\email{jialongchen1@163.com(J.L. C.)}
\email{xyyan1103@163.com(X.Y. Y.)}
\date{\today}
\thanks{*Corresponding author.}

\begin{abstract}
	Let the infinite convolutions
\begin{equation*}
	\mu_{\{R_{k}\},\{D_{k}\}}=\delta_{R_{1}^{-1}D_{1}}*\delta_{R_{1}^{-1}R_{2}^{-1}D_{2}}*\delta_{R_{1}^{-1}R_{2}^{-1}R_{3}^{-1}D_{3}}*\dotsi
\end{equation*} 
	be generated by the sequence of pairs \(\{ (R_k,D_k) \}_{k=1}^{\infty} \), where  $ R_k\in M_n(\mathbb{Z})$ is an expanding integer matric, $D_k$ is a finite integer  digit sets that satisfies the following two conditions: (i). \( \# D_k = m \)  and \(   m>2 \) is  a  prime; (ii). \( \{x: \sum_{d\in D_{k}}e^{2\pi i\langle d,x \rangle}=0\} =\cup_{i=1}^{\phi(k)}\cup_{j=1}^{m-1}(\frac{j}{m}\nu_{k,i}+\mathbb{Z}^{n}) \) for some \( \nu_{k,i} \in \{ (l_1, \cdots, l_n)^t : l_i \in [1, m-1] \cap \mathbb{Z}, 1\leq i\leq n \} \).  In this paper, we study the  spectrality  of $\mu_{\{R_{k}\},\{D_{k}\}}$, and some necessary and sufficient conditions for \( L^{2}(\mu_{\{R_{k}\},\{D_{k}\}}) \) to have an orthogonal exponential function basis are established. Finally, we discuss the explanations and applications of our results.

\end{abstract}
\maketitle 
\section{\bf  Introduction}
Let \( \mu \) be a Borel probability measure with compact support on \( \mathbb{R}^{n} \).
If there exists a countable set $\Lambda\subset\mathbb{R}^{n}$ such that the family of exponential functions $E_{\Lambda} = \{ e^{-2\pi i \langle \lambda, x \rangle} : \lambda \in \Lambda \}$ forms an orthonormal basis for $ L^{2}(\mu) $, where \(\langle \cdot, \cdot \rangle\) denotes the standard inner product on \(\mathbb{R}^{n}\), then \(\mu\) is called a spectral measure with spectrum \(\Lambda\). The pair \((\mu, \Lambda)\) is referred to as a spectral pair. For the special case that  a spectral measure is the Lebesgue measure on a measurable set $\Omega$,  we say that $\Omega$ is a spectral set. It is well-known from classical Fourier analysis that the unit cube $[0,1]^{n}$ is a spectral set with spectrum $\mathbb{Z}^{n}$. 

The spectral measure is a natural generalization of the spectral set introduced by Fuglede \cite{B1}, who proposed \textbf{Fuglede's conjecture}: \(\Omega \subset \mathbb{R}^n\) is a spectral set if and only if \(\Omega\) is a translational tile. The conjecture was finally disproved by Tao \cite{T1} and others \cite{KM1} in $\mathbb{R}^{n}$, $n\ge3$. Recently, Lev
and Matolcsi \cite{LM} showed that the spectral set conjecture holds in all dimensions for convex domains. Generally speaking, a probability measure needs strict conditions to become a spectral measure, and only a limited number of spectral measures have been discovered or constructed.  Therefore, a fundamental problem in this area is as follows: \textbf{Find or construct a spectral measure$!$}

It is known \cite{HLL1} that a spectral measure $\mu$ must be of pure type: $\mu$ is absolutely continuous or singular continuous with respect to the Lebesgue measure or counting measure supported on a finite set. In 1998, Jorgensen and Pedersen \cite{JP1} discovered the first example of spectral measure which is non-atomic and singular to the Lebesgue measure. They proved that the one fourth Cantor measure $\mu_{4}$, which satisfies
\begin{equation*}
	\mu_{4}(X)=\frac{1}{2}\mu_{4}(4X)+\frac{1}{2}\mu_{4}(4X-2)\;\;\text{for all Borel set}\; X\subset \mathbb{R}
\end{equation*}
and with the compact support $T=\{\sum_{i=1}^{\infty}4^{-i}d_{i}:d_{i}\in\{0,2\}\}$, is a spectral measure with a spectrum
\begin{equation*}
	\Lambda=\{\sum_{i=1}^{k}4^{i-1}d_{i}: d_{i}\in\{0,1\},k\in\mathbb{N}\}.
\end{equation*}
Since then, research on singular spectral measures has flourished, particularly within fractal theory, as evidenced by studies on various specific fractal spectral measures \cite{AFL1,CYZ,DH1,DHL1,LDL,C,DHL,FHW,HH1,TY1,Y1,LDZ,DL1,DHLY}. A key strategy to study the spectra theory of fractal measures such as self-affine measure and Moran measure, is by using infinite convolutions. 
Given a sequence $\{(R_k, D_k)\}_{k=1}^\infty$ where
\( \{R_{k}\}_{k=1}^{\infty} \subset M_{n}(\mathbb{Z}) \) and \( \{D_{k}\}_{k=1}^{\infty} \subset \mathbb{Z}^{n} \) is finite for all $k \in \mathbb{N}$. 
Then we write
\[
\mu_k = \delta_{R_1^{-1}D_1} \ast \delta_{R_1^{-1}R_2^{-1}D_2} \ast \cdots \ast \delta_{R_1^{-1}R_{2}^{-1}\cdots R_k^{-1}D_k}.
\]
If the sequence $\{\mu_k\}_{k=1}^\infty$ converges weakly to a Borel probability measure $\mu$, 
then we call $\mu$ the infinite convolution of $\{(R_k, D_k)\}_{k=1}^\infty$, denoted by
\begin{equation}\label{1.1}
\mu := \mu_{\{R_{k}\},\{D_{k}\}} = \delta_{R_{1}^{-1}D_{1}} * \delta_{R_{1}^{-1}R_{2}^{-1}D_{2}} *  \cdots \ast \delta_{R_1^{-1}R_2^{-1}\cdots R_k^{-1}D_k} \ast \cdots.\end{equation}
Here, \( \delta_{E} = \frac{1}{\#E} \sum_{a \in E} \delta_{a} \), where \( \delta_{a} \) is the Dirac probability measure with mass at the point \( a \). This measure is known as a Riesz product measure or Moran measure, and it is supported on the following discrete set:
\begin{equation*}
T(\{R_{k}\}, \{D_{k}\}) = \{ \sum_{k=1}^{\infty} (R_{k} R_{k-1} \cdots R_{1})^{-1} d : d \in D_{k} \} = \sum_{k=1}^{\infty} (R_{k} R_{k-1} \cdots R_{1})^{-1} D_{k}.
\end{equation*}
Hadamard triples are fundamental to the study of spectral infinite convolutions. Let $R\in M_{n}(\mathbb{Z})$  be an $n\times n$ expanding matrix, and let $ D,L \subset \mathbb{Z}^{n} $ be finite sets with $ \#D = \#L $. We say that the system \( (R, D, L) \) forms a Hadamard triple $($or that \( (R^{-1}D, L) \) is a compatible pair$)$ if
\begin{equation*}
	H=\frac{1}{\sqrt {\#D}}[e^{2\pi i  \langle R^{-1}d, l  \rangle}]_{d\in D, l\in L}
\end{equation*} is unitary, i.e., $HH^*=I$, where $H^{*}$ denotes the transposed conjugate of $H$. It is very convenient to construct an orthogonal family of exponential functions for \( L^{2}(\mu) \) using a Hadamard triple. However, the challenge lies in verifying that the constructed set \( \Lambda \) is indeed an orthonormal basis for \( L^{2}(\mu) \). Recent studies increasingly reveal a significant connection between Hadamard triples and spectral measures \cite{LW1, S, S1}.

This article primarily investigates the spectral properties of a class of Sierpinski-type Moran measures on \( \mathbb{R}^{n} \) and also presents a criterion for the spectral properties of a class of Moran measures on \(\mathbb{R}^{2} \). There are multiple motivations for this research. The standard Sierpinski self-affine measure is supported on the Sierpinski gasket \cite{W}, a celebrated fractal set that constitutes an important object of study in fractal analysis \cite{J}. Moreover, it is of considerable significance in the investigation of spectral measures. It is known that the zeros of the Fourier transform of a measure play an important role in orthogonal harmonic analysis, and Sierpinski-type measures represent a class of measures with discrete zeros of the Fourier transform.

In the following, we present some necessary concepts. For any $ R\in M_{n}(\mathbb{Z}) $, we define the norm as
$
 \| R\|^{'}=\sup_{x\neq0}\dfrac{\|Rx\|}{\| x\|} ,
$
where $ \|\cdot\| $ denotes the Euclidean norm.
For a finite set \( D \subset \mathbb{Z}^{n} \), define the mask polynomial of \( D \) by
\begin{equation*}
m_{D}(\xi) = \frac{1}{\#D} \sum_{d \in D} e^{2\pi i \langle \xi, d \rangle}, \quad \xi \in \mathbb{R}^{n}.
\end{equation*}
 Define $\mathcal{Z}(m_{D})=\{\xi\in \mathbb{R}^{n}:m_{D}(\xi)=0\}.$ Let \( \{R_k\}_{k=1}^{\infty} \) be a sequence of expanding matrices in \( M_n(\mathbb{Z}) \), and let \( \{D_k\}_{k=1}^{\infty} \) be a sequence of integer digit sets in \( \mathbb{Z}^n .\) Assuming
\begin{equation}\label{1.2}
\limsup_{k\rightarrow\infty} \| R_{k}^{-1}\|^{'}\le r<1\;\; \text{and}\;\;  \sup\{\|d\|:d\in D_{k},k\ge1\}<\infty,
\end{equation}
this guarantees the existence of the Moran measure \( \mu_{\{R_{k}\},\{D_{k}\}} \) defined by \eqref{1.1}. For \( k \ge 1 \), assume that \( \# D_k = m \) and the zero set \( \mathcal{Z}(m_{D_k}) \) can be decomposed into a finite union of disjoint sets, i.e.,
\begin{equation}\label{eq1.2}
	\mathcal{Z}(m_{D_k}) = \cup_{i=1}^{\phi(k)} \mathcal{Z}_{k,i}(m),
\end{equation}
where \( m \) is a prime and \( \mathcal{Z}_{k,i}(m) \) satisfies
\begin{equation}\label{eq1.5}
	\mathcal{Z} _{k,i}(m)=\cup_{j=1}^{m-1}(\frac{j}{m}\nu_{k,i}+\mathbb{Z}^{n})
\end{equation}
	for some $ \nu_{k,i}\in\{(l_{1},\cdots,l_{n})^{t}: l_{1},\cdots,l_{n}\in [1,m-1]\cap\mathbb{Z}\} $, where \( \phi(\cdot) \) is a mapping from \( \mathbb{Z}^+ \) to \( \mathbb{Z}^+ \).

The class of models that satisfies \eqref{eq1.2} represents a natural generalization of the Sierpinski-type family, which plays a significant role in fractal geometry and geometric measure theory \cite{F1}. Many researchers have studied the model in \eqref{eq1.2} and have derived a series of results, as discussed in \cite{CLZ,DHL1,DL, WLS1, Y1}. Among these, the self-similar measure \( \mu_{\rho,D_{1}} \) studied by Dai  \cite{DHL1}, where \( \rho \in \mathbb{R} \) and \( D_{1} = \{0, 1, \cdots, N-1\} \subset \mathbb{Z} \), as well as the self-affine measure \( \mu_{R,D_{2}} \) investigated by Deng et al.  \cite{DL}, where \( R = \text{diag}[\rho, \rho] \in M_{2}(\mathbb{R}) \) and \( D_{2} = \{(0,0)^{t}, (1,0)^{t}, (0,1)^{t}\} \).
By direct calculation, we obtain
\begin{equation*}
	\mathcal{Z}(m_{D_{1}})=\{\frac{j}{N}:j=1,2,\cdots,N-1\}+\mathbb{Z} \;\;\text{and} \;\; \mathcal{Z}(m_{D_{2}})=\dfrac{1}{3}\{ (1,2)^{t} , (2,1)^{t}\}+\mathbb{Z}^{2},
\end{equation*}
both of which satisfy condition \eqref{eq1.2}.
There are also additional digit sets that satisfy the model in \eqref{eq1.2}, such as $\{(0,0)^{t}, (1,0)^{t}, (0,-1)^{t}\}$, $\{(0,0)^{t}, (1,0)^{t}, (1,1)^{t}\}$  and $\{(0,0)^{t}, (2,3)^{t}, (3,5)^{t}\}$. Next, we will impose certain restrictions on the matrix \( R_k \). There exist $ \delta,\beta \in (0,\frac{1}{4}) $ and \( N \) such that for any \( p \ge 0 \) and \( k > N \),
\begin{equation}\label{a}
	R_{k+1}^{t}R_{k+2}^{t}\dotsi R_{k+p}^{t}\in \mathcal{A}_{\delta,\beta},
\end{equation}
where
\begin{equation*}
	\mathcal{A}_{\delta,\beta}:=\{A\in M_{n}(\mathbb{Z}): A^{-1}[-\frac{1}{2}-\delta,\frac{1}{2}+\delta]^{n}\cap (\cup_{k=1}^{\infty}\mathcal{Z}(m_{D_{k}}))_{\beta}=\emptyset\},
\end{equation*}
and  \( E_{\beta} = \{ x : \sup_{y \in E} \| x - y \| < \beta \} \) is the \( \beta \)-neighborhood of \( E \) under the Euclidean norm \( \| \cdot \| \). Using the notation and concepts defined above, we now present our main results.
\begin{thm}\label{thm1.4}
Let \( \mu_{\{R_k\}, \{D_k\}} \) be defined by \eqref{1.1} and \eqref{eq1.2}, and satisfy the conditions \eqref{1.2}, \eqref{a}. If there exists $i_{k}\in\{1,\cdots,\phi(k)\}$ such that $R_{k}^{t}\nu_{k,i_{k}} \in m\mathbb{Z}^{n}\;\text{for}\;k\ge 2. $ Then $ \mu_{\{R_{k}\},\{D_{k}\}} $ is a spectral measure.
\end{thm}

This theorem extends results from \cite{DHLY}, and we cite an example from \cite{LLZ} to illustrate that the condition \eqref{1.2} is essential. We also conjecture that this result is necessary; however, we have not yet found a suitable method to prove it. Surprisingly, we obtain the following result when \( R_k \) is a diagonal matrix.
\begin{thm}\label{thm1.2}
Let $ \mu_{\{R_{k}\},\{D_{k}\}} $ be  defined by \eqref{1.1} and \eqref{eq1.2}, satisfying condition \eqref{1.2}, where \( R_{k} = \text{diag} [p_{k,1}, \cdots, p_{k,n}] \) for \( k \ge 1 \).  If \( m > 2 \), then \( \mu_{\{R_{k}\}, \{D_{k}\}} \) is a spectral measure if and only if \( m \mid p_{k,i} \) for \( k \ge 2 \) and \( i = 1, 2, \cdots, n \).
\end{thm}
If \( R_k \) is not a diagonal matrix for \( k \geq 1 \), and we restrict \( \phi(k) = 1 \) to be a constant function, we obtain the following necessary and sufficient condition for the measure \( \mu_{\{R_k\}, \{D_k\}} \) to be a spectral measure.
\begin{thm}\label{thm1.3}
Let \( \mu_{\{R_k\}, \{D_k\}} \) be defined by \eqref{1.1} and \eqref{eq1.2}, and satisfy the conditions \eqref{1.2}, \eqref{a}. If \( m > 2 \) and \( \phi(k) = 1 \) for \( k \geq 1 \), then \( \mu_{\{R_{k}\}, \{D_{k}\}} \) is a spectral measure if and only if \(  R_{k}^{t}\nu_{k}  \in m\mathbb{Z}^{n} \) for \( k \geq 2 \).
\end{thm}
It should be noted that \( \nu_{k} \) is also given by \eqref{eq1.5}. However, the assumption  \( \phi(k) = 1 \) implies that the second index \( i \) is unnecessary. In fact, the sufficiency of  Theorems \ref{thm1.2} and  \ref{thm1.3} can be derived from Theorem \ref{thm1.4}; therefore, our main focus in the subsequent proofs will be on their necessity. This result complements the findings presented in \cite{CCWW}.

Furthermore, we present the following corollaries, which are important results of the paper.
\begin{cor}\label{cor1.5}
Let $ \mu_{\{R_{k}\},\{D_{k}\}} $ be  defined by \eqref{1.1} and \eqref{eq1.2}, satisfying condition \eqref{1.2}, where
\begin{equation}\label{1.7}
R_{k} := \begin{bmatrix}
	a_{1}^{(k)} & a_{1}^{(k)} & \cdots & a_{1}^{(k)} \\
	0 & a_{2}^{(k)} & \cdots & a_{2}^{(k)} \\
	\vdots & \vdots & \ddots & \vdots \\
	0 & 0 & \cdots & a_{n}^{(k)}
\end{bmatrix} \in M_{n}(\mathbb{Z}) \quad \text{for } k \geq 1.
\end{equation}
Suppose that \( m > 2 \) is a prime and \( \phi(k) = 1 \)  for \( k \geq 1 \). Then \( \mu_{\{R_{k}\},\{D_{k}\}} \) is a spectral measure if and only if \( m \mid a_{i}^{(k)} \) for \( k \geq 2 \) and \( i \in \{1, 2, \cdots, n\} \).
\end{cor}
The following corollary examines the Sierpinski-type measure on \( \mathbb{R}^2 \). Write
\begin{equation*}
 \Gamma_{1}:=\{\{(0,0)^{t}, (a,b)^{t}, (c,d)^{t}\}:(-d-c)=(a+b)\pmod3\}
\end{equation*}
and
\begin{equation*}
	\Gamma_{2}:=\{\{(0,0)^{t}, (a,b)^{t}, (c,d)^{t}\}:(d-c)=(a-b) \pmod3\}.
\end{equation*}
\begin{cor}\label{cor1.6}
Let $ \mu_{\{R_{k}\},\{D_{k}\}} $ be  defined by \eqref{1.1}, and satisfy the conditions \eqref{1.2}, \eqref{a}, where
\begin{equation*}
		 D_{k}=\{(0,0)^{t}, (a_{k},b_{k})^{t}, (c_{k},d_{k})^{t}\}\in \Gamma_{1}\cup\Gamma_{2} \;\;\text{and}\;\; |a_{k}d_{k}-b_{k}c_{k}|=1\;\;\text{for}\;\;k\ge1.
\end{equation*}
Then \( \mu_{\{R_{k}\},\{D_{k}\}} \) is a spectral measure if and only if $ R_{k}^{t}(1,i)^{t}\in 3\mathbb{Z}^{2} $ for $D_{k}\in\Gamma_{i}$, $i\in\{1,2\}$ and $ k \ge 2 $.
\end{cor}

The organization of the paper is as follows. In Section 2, we introduce some notation and basic lemmas related to spectral measures that will be used throughout the paper. In Section 3, we prove Theorem \ref{thm1.4}. In Section 4, we primarily establish the necessity of Theorem \ref{thm1.2}, while its sufficiency can be derived from Theorem \ref{thm1.4}.  In Section 5, we first prove Theorem \ref{thm1.3}, Corollary \ref{cor1.5} and Corollary \ref{cor1.6}, and then provide some examples to illustrate our results.

\section{\bf Preliminaries}
Let $\mu$ be a probability measure with compact support in $\mathbb{R}^{n}$, its Fourier transform is defined as usual,
   $$
   \widehat{\mu}(\xi)=\int_{\mathbb{R}^{n}}e^{-2\pi i  \langle\xi, x\rangle } d\mu(x),\quad \xi\in\mathbb{R}^{n} .
   $$
It is easy to show that $\Lambda$ is an orthogonal set of $\mu$ if and only if $\widehat{\mu}(\lambda-\lambda ')=0$ for any $\lambda\neq \lambda '\in \Lambda$. In other word, $E_{\Lambda}=\{e^{-2\pi i \langle\lambda, x\rangle}:\lambda\in \Lambda\}$ is an orthogonal family of $ L^2(\mu) $ if and only if
 $ (\Lambda-\Lambda)\backslash\{0\}\subset\mathcal{Z}( \widehat{\mu})  $. Define the function
   $$
   Q_{\mu,\Lambda}(\xi)=\displaystyle{\sum_{\lambda\in \Lambda}}|\widehat{\mu}(\xi+\lambda)|^2.
   $$
 The following theorem is a basic criterion for the spectrality of $\mu$ \rm{\cite{JP1}}.
    \medskip
 \begin{lemma}[\cite{JP1}]\label{lem2.2}
    Let $\mu$ be a Borel probability measure with compact support in $\mathbb{R}^{n}$, and let $ \Lambda \subset \mathbb{R}^{n} $ be a countable subset. Then

    $\rm (i)$ $E_{\Lambda}$ is an orthogonal family of $L^2(\mu)$ if and only if $Q_{\mu,\Lambda}(\xi)\leq1$ for $\xi\in \mathbb{R}^{n}$;

    $\rm (ii)$ $E_{\Lambda}$ is an orthogonal basis for $L^2(\mu)$ if and only if $Q_{\mu,\Lambda}(\xi)=1$ for $\xi\in \mathbb{R}^{n}$;

   $\rm (iii)$ $Q_{\mu,\Lambda}(\xi)$ has an entire analytic extension to $\mathbb{C}^{n}$ if $\Lambda$ is an orthogonal set of $\mu$.
    \end{lemma}
In the following, we introduce a simple yet very useful conclusion regarding weighted sums.
 \begin{lemma}[\cite{DL1}]\label{lem2.1}
 Let $ p_{i,j} $ be positive numbers such that $ \sum_{j=1}^{k}p_{i,j}=1 $ $( i=1,\;2,\dotsi,\;m) $ and $ x_{i,j} $ be non-negative numbers with $ \sum_{i=1}^{m}\max_{1\le j\le k}x_{i,j}\le1 $. Then $ \sum_{i=1}^{m}\sum_{j=1}^{k}p_{i,j}x_{i,j}=1 $ if and only if $ x_{i,1}=\dotsi=x_{i,k} $ for $ 1\le i\le m $ and $ \sum_{i=1}^{m}x_{i,1}=1 $.
 \end{lemma}
The following is a standard result for compatible pairs \cite{LW1, S, S1}.

 \begin{pro}\label{lem2.4}
 Let $ S,D\subset \mathbb{Z}^{n} $ and let $ R\in M_{n}(\mathbb{Z}) $ such that $ (R^{-1}S,D) $ be an integral compatible pair. Then the following statements hold:

	$\rm(i)$ $ R^{-1}S $ is a spectrum of the measure $ \delta_{D} $;
	
	$\rm(ii)$ $ (R^{-1}(S+s),D+d) $ for $ s,d\in \mathbb{Z}^{n} $ and $ (R^{-1}S,-D) $ are integral compatible pairs;
	
    $\rm(iii)$ $ (R^{-k}S_{k},D_{k}) $ is an integral compatible pair for $ k\ge1 $, where $ S_{k}=S+RS+\dotsi+R^{k-1}S $
and $ D_{k}=D+R^{t}D+\dotsi+({R^{t}})^{k-1}D $;

    $\rm(iv)$ All elements in $ S $ $(resp.D)$ are in different coset of the group $ \mathbb{Z}^{n}/R\mathbb{Z}^{n} $ $($resp. $ \mathbb{Z}^{n}/R^{t}\mathbb{Z}^{n} $$)$;

	$\rm(v)$ Suppose that $ \widetilde{D}$, $ \widetilde{S}\subset\mathbb{Z}^{n} $ such that $ \widetilde{D}=D\pmod R^{t} $ and $ \widetilde{S}=S\pmod R $. Then
$ (R^{-1}\widetilde{S},\widetilde{D}) $ is an integral compatible pair;
	
	$\rm(vi)$ Suppose that all $ (R_{k}^{-1}S_{k},D_{k}) $ are integral compatible pairs for $ k\ge1 $. Then
\begin{equation*}
((R_{k}R_{k-1}\dotsi R_{1})^{-1}\widetilde{S}_{k},\widetilde{D}_{k})
\end{equation*}
	is an integral compatible pair for each $ k\ge1 $, where
\begin{equation*}
 \widetilde{D}_{k}=D_{1}+R_{1}^{t}D_{2}+\dotsi+R_{1}^{t}R_{2}^{t}\dotsi R_{k-1}^{t}D_{k},\quad\widetilde{S}_{k}=S_{k}+R_{k}S_{k-1}+\dotsi+R_{k}R_{k-1}\dotsi R_{2}S_{1}.
\end{equation*}
 \end{pro}
 From now on, for convenience, we  define
\begin{align*}
	&\mu:=\mu_{\{R_{k}\},\{D_{k}\}},\quad\mu_{l}:=\delta_{R_{1}^{-1}D_{1}}*\delta_{R_{1}^{-1}R_{2}^{-1}D_{2}}*\dotsi*\delta_{R_{1}^{-1}R_{2}^{-1}\dotsi R_{l}^{-1}D_{l}},\\
	&\mu_{l<k\le h}:=\delta_{R_{l+1}^{-1}D_{l+1}}*\delta_{R_{l+1}^{-1}R_{l+2}^{-1}D_{l+2}}*\dotsi*\delta_{R_{l+1}^{-1}R_{l+2}^{-1}\dotsi R_{h}^{-1}D_{h}},\\
	&\mu_{>l}:=\delta_{R_{1}^{-1}\cdots R_{l+1}^{-1}D_{l+1}}*\delta_{R_{1}^{-1}\cdots R_{l+1}^{-1}R_{l+2}^{-1}D_{l+2}}*\cdots,\\
&\mu_{k>l}:=\delta_{R_{l+1}^{-1}D_{l+1}}*\delta_{R_{l+1}^{-1}R_{l+2}^{-1}D_{l+2}}*\cdots
\end{align*}
and
\begin{equation}\label{2.2}
 \mu_{\{R\}_{k>l}}:=\mu_{k>l}\circ(R_{l+1}^{-1}R)=\delta_{R^{-1}D_{l+1}}*\delta_{R^{-1}R_{l+2}^{-1}D_{l+2}}*\dotsi.
\end{equation}
Thus
\begin{equation*}
\mu_{>l}=\mu_{k>l}\circ(R_{l}R_{l-1}\cdots R_{1})\;\text{and} \; 	
\mu=\mu_{l}*\mu_{>l}=\mu_{l}*\mu_{k>l}\circ(R_{l}R_{l-1}\cdots R_{1}).
    \end{equation*}
The following lemma tells us that the spectrality of $\mu$ is invariant under a linear transformation.
  \begin{lemma}\label{lem2.6}
  Let $ R\in M_{n}(\mathbb{Z}) $ be a nonsingular matrix, and let $ \mu_{\{R\}_{k>0}} $ be defined by  \eqref{2.2}. Then $ \mu=\mu_{\{R\}_{k>0}}\circ(R^{-1}R_{1}) $ and $ ( \mu,\Lambda) $ is a spectral pair if and only if $ \mu_{\{R\}_{k>0}} $ is a spectral measure with spectrum $ {R^{t}R_{1}^{t}}^{-1}\Lambda $.
  \end{lemma}
At the end of this section, we introduce the concept of equicontinuity. A family of functions \( F \subset C(\mathbb{R}^{n}) \) is called equicontinuous if for any \( \epsilon > 0 \), there exists \( \delta > 0 \) such that whenever \( \|x - y\| \le \delta \), we have \( |f(x) - f(y)| < \epsilon \) for all \( f \in F \), where \( \|\cdot\| \) denotes the Euclidean norm on \( \mathbb{R}^{n} \). The following lemma is also well-known.
\begin{lemma}\label{lem2.8}
Let \( I \) be a compact set in \( \mathbb{R}^{n} \). Then the set \( F(I) \), which denotes the Fourier transforms of all Borel probability measures supported on the compact set \( I \subset \mathbb{R}^{n} \), is equicontinuous.
\end{lemma}

\section{\bf A sufficient condition of spectral measure $\mu_{\{R_{k}\}, \{D_{k}\}}$}
 In this section, we prove Theorem \ref{thm1.4}. First, we introduce some necessary concepts. For any $ R\in M_{n}(\mathbb{Z}) $, recall that
$ \| R\|^{'}=\sup_{x\neq0}\dfrac{\|Rx\|}{\| x\|} ,$
where $ \|\cdot\| $ denotes the Euclidean norm. In \( \mathbb{R}^{n} \), the matrix norm \( \| \cdot\|^{'} \) and the Euclidean norm \( \| \cdot\| \) are equivalent. Therefore, there exists a constant $ c\ge1 $ such that
$c^{-1}\|\cdot\|^{'}\le\|\cdot\|\le c\|\cdot\|^{'}.$
Hence,
\begin{equation}\label{3.2}
\| {R_{j}^{t}}^{-1}\|^{'}\le r,\;1\le j\le k \Longrightarrow \| {R_{1}^{t}}^{-1}{R_{2}^{t}}^{-1}\dotsi {R_{k}^{t}}^{-1}\|^{'}\le c^{2} r^{k}.
\end{equation}
For that purpose, we need the following two technical lemmas.
\begin{lemma}\label{lem3.1}
Let \( \{ R_k \}_{k=1}^{\infty} \subset M_n(\mathbb{Z} ) \) satisfy \eqref{a} and \( \| R_k^{-1} \|' \le r < 1 \) for \( k \ge 1 \). There exist constants \( M \) and \( \gamma \), as well as a monotonic increasing positive sequence \( \gamma_{i} \) that converges to 1, such that
\begin{equation*}
|\widehat{\mu}_{k>q}(\xi)| \geq \gamma_{i} \gamma^{i}
\end{equation*}
for all \( i \geq M \), \( q \geq N \) and \( \xi \in [-\frac{1}{2}-\delta, \frac{1}{2}+\delta]^{n} \), where $N$ and $\delta$ are given by \eqref{a}.
\end{lemma}
\begin{proof}
We can easily decompose
$|\widehat{\mu}_{k>q}(\xi)|$ into the product of the following two parts:
\begin{equation}\label{3.3}
 |\widehat{\mu}_{k>q}(\xi)|=|\widehat{\mu}_{q<k\le q+i}(\xi)||\widehat{\mu}_{k>q+i}({R_{q+i}^{t}}^{-1}\dotsi {R_{q+1}^{t}}^{-1}\xi)|.
\end{equation}
Next, we will estimate these two parts. From \eqref{3.2}, it follows that $  \| {R_{q+i}^{t}}^{-1}\dotsi {R_{q+1}^{t}}^{-1}\|^{'} \le c^{2}r^{i} $. Since $ 0<\delta<\frac{1}{4} $, one gets that
\begin{equation}\label{3.4}
 \| {R_{q+i}^{t}}^{-1}\dotsi {R_{q+1}^{t}}^{-1}\xi\|\le c^{2}r^{i}\|\xi\| \le \frac{3\sqrt{n}}{2}c^{2}r^{i}<2\sqrt{n}c^{2} \;\;\;\text{for}\; \xi \in[-\frac{1}{2}-\delta,\frac{1}{2}+\delta]^{n}.
\end{equation}
Hence, for any $ q\ge N $ and $ i\ge 1 $, it follows that
\begin{equation*}
{R_{q+i}^{t}}^{-1}\dotsi {R_{q+1}^{t}}^{-1}[-\frac{1}{2}-\delta,\frac{1}{2}+\delta]^{n}\subset B(0,2\sqrt{n}c^{2})\backslash(\cup_{j=1}^{\infty}\mathcal{Z}(m_{D_{j}}))_{\beta}.
\end{equation*}
For each $k$, we define
$\eta_{k}:=\inf\{|m_{D_{k}}(x)|:x\in B(0,2\sqrt{n}c^{2})\backslash(\cup_{j=1}^{\infty}\mathcal{Z}(m_{D_{j}}))_{\beta}\}.$
Clearly, $\eta_{k}>0. $ Then according to $\sup\{\|d\|:d\in D_{k},\;k\ge1\}<\infty$, we get $\gamma:=\min\{\eta_{k}:k\ge1\}>0.$
Therefore,
\begin{equation}\label{3.5}
 |\widehat{\mu}_{q<k\le q+i}(\xi)|=\prod_{j=1}^{i}|m_{D_{k}}({R_{q+j}^{t}}^{-1}\dotsi {R_{q+1}^{t}}^{-1}\xi)|\ge \gamma^{i}
\end{equation}
for any $q\ge N$, $i\ge 1$ and $\;\xi \in [-\frac{1}{2}-\delta,\frac{1}{2}+\delta]^{n}.$

Now, let us consider the latter term of $\eqref{3.3}$. Note that $\sup\{\|d\|:d\in D_{k},\;k\ge1\}<\infty$, we have that $  \cup_{k=1}^{\infty}\{D_{k}\} $ is a compact set. It follows from Lemma \ref{lem2.8} that the set $ \{m_{D_{k}}:k\ge1\}  $ is equicontinuous. By the continuity of $ m_{D_{k}}(x) $, as stated in \eqref{3.4}, and considering that \( m_{D_{k}}(0) = 1 \), there exists an \( M \) that depends only on \( c \) and \( r \) such that
\begin{equation*}
|m_{D_{k}}({R_{q+i}^{t}}^{-1}\dotsi {R_{q+1}^{t}}^{-1}\xi)| \ge \frac{1}{2}
\end{equation*}
for all $k\ge1$, $ i\ge M,\;q\ge N ,\;\text{and}\; \xi\in [-\frac{1}{2}-\delta,\frac{1}{2}+\delta]^{n}.$ Write $ s:=\max\{\| d\|:d\in \cup_{k=1}^{\infty}D_{k}\} $. It is easy to check that $ |m_{D_{k}}(x)-1|\le 2s\pi\| x\| $ and $ -\ln x\le2(1-x)\; \text{for}\;\frac{1}{2}\le x\le1.$ Then,
\begin{align*}
-\ln|\widehat{\mu}_{k>q+i}({R_{q+i}^{t}}^{-1}\dotsi {R_{q+1}^{t}}^{-1}\xi)|=&\sum_{j=1}^{\infty}-\ln|m_{D_{q+i+j}}({R_{q+i+j}^{t}}^{-1}\dotsi {R_{q+1}^{t}}^{-1}\xi)|\\
\le& 2\sum_{j=1}^{\infty}|1-m_{D_{q+i+j}}({R_{q+i+j}^{t}}^{-1}\dotsi {R_{q+1}^{t}}^{-1}\xi)|\\
\le& 4\pi s\sum_{j=1}^{\infty}\| {R_{q+i+j}^{t}}^{-1}\dotsi {R_{q+1}^{t}}^{-1}\xi\| \le 4\sqrt{n}c^{2}\pi s\frac{r^{i+1}}{1-r}.
\end{align*}
This implies
\begin{equation}\label{3.6}
 |\widehat{\mu}_{k>q+i}({R_{q+i}^{t}}^{-1}\dotsi {R_{q+1}^{t}}^{-1}\xi)|= \prod_{j=1}^{\infty}|m_{D_{k}}({R_{q+i+j}^{t}}^{-1}\dotsi {R_{q+1}^{t}}^{-1}\xi)|\ge\gamma_{i},
\end{equation}
where $ \gamma_{i}=\exp\{-4\sqrt{n}c^{2}\pi s\frac{r^{i+1}}{1-r}\} $. Clearly $ \gamma_{i}\rightarrow1 $ as $ i\rightarrow\infty $. Combining \eqref{3.5} and \eqref{3.6}, we complete the proof.
\end{proof}
The crucial step in demonstrating that \( \mu \) is a spectral measure is constructing its spectrum. This construction method is based on Lemma \ref{lem2.4}. For some suitable integers $ K $, $ k\ge0 $, we write
\begin{align*}
 \widetilde{D}_{k}=&D_{(k+1)K}+R_{(k+1)K}D_{(k+1)K-1}+R_{(k+1)K}R_{(k+1)K-1}D_{(k+1)K-2}+\dotsi\\ &+R_{(k+1)K}R_{(k+1)K-1}\dotsi R_{kK+2}D_{kK+1}
 \end{align*}
and
\begin{equation*}
 \widetilde{R}_{k}=R_{(k+1)K}R_{(k+1)K-1}\dotsi R_{kK+1}.
\end{equation*}
It is not hard to see that
\begin{equation*}
 \mu=\delta_{\widetilde{R}_{0}^{-1}\widetilde{D}_{0}}*\delta_{\widetilde{R}_{0}^{-1}\widetilde{R}_{1}^{-1}\widetilde{D}_{1}}*\dotsi*\delta_{\widetilde{R}_{0}^{-1}\widetilde{R}_{1}^{-1}\dotsi \widetilde{R}_{k}^{-1}\widetilde{D}_{k}}*\dotsi
\end{equation*}
and
\begin{equation*}
 |\widehat{\mu}(\xi)|=\prod_{k=0}^{\infty}|m_{\widetilde{D}_{k}}((\widetilde{R}_{0}^{t}\dotsi \widetilde{R}_{k}^{t})^{-1}\xi)|.
\end{equation*}
Let $ c_{k,i}^{(l)}=l\nu_{k,i}\pmod {m\mathbb{Z}^{n}} $ such that $ \frac{c_{k,i}^{(l)}}{m}\subset (-\frac{1}{2},\frac{1}{2}]^{n} $, where $ l\in\{1,2,\cdots,m-1\} $, $i\in\{1,\cdots,\phi(k)\}$. Furthermore, we can denote $ C_{k,i}=\{0,c_{k,i}^{(1)},c_{k,i}^{(2)},\cdots,c_{k,i}^{(m-1)}\} .$ For each \( k \) and any \( i\in\{1,\cdots,\phi(k)\} \), $ (R_{k}^{-1}D_{k},\frac{1}{m}R_{k}^{t}C_{k,i}) $ is compatible pair.

 \begin{lemma}\label{lem3.2}
 Suppose that for each \( k\ge1 \), $ \| R_{k}^{-1}\|^{'}\le r$ and there exists \( i_{k}\in\{1,\cdots,\phi(k)\} \) such that $R_{k}^{t}\nu_{k,i_{k}}\in m\mathbb{Z}^{n}$. Then there exist $ L_{k}\subset \mathbb{Z}^{n} $ with $ 0\in L_{k} $ and $ K\ge M $ such that $ (\widetilde{R}_{k}^{-1}\widetilde{D}_{k},L_{k}) $ is an integral compatible pair, and
\begin{equation}\label{3.7}
(\widetilde{R}_{0}^{t}\dotsi \widetilde{R}_{k}^{t})^{-1}L_{0}+\dotsi+(\widetilde{R}_{k-1}^{t}\widetilde{R}_{k}^{t})^{-1}L_{k-1}+(\widetilde{R}_{k}^{t})^{-1}L_{k}\subset[-\frac{1}{2}-\dfrac{1}{4}\delta,\frac{1}{2}+\dfrac{1}{4}\delta]^{n}
\end{equation}
 for $ k\ge0, $ where $ M $ is given by Lemma \ref{lem3.1} and $\delta$ is defined by \eqref{a}.
 \end{lemma}
 \begin{proof}
 Let $ N_{k}=\widetilde{R}_{k}^{t}(-\frac{1}{2},\frac{1}{2}]^{n}\cap\mathbb{Z}^{n}$ for $ k\ge0 $. Then \( N_{k} \) is a complete residue set modulo \( \widetilde{R}_{k}^{t} \).
 Similarly, let $ E_{k}=R_{k}^{t}(-\frac{1}{2},\frac{1}{2}]^{n}\cap\mathbb{Z}^{n}$. It is easy to show that
\begin{equation*}
 E_{kK+1} + R_{kK+1}^{t} E_{kK+2} + \cdots + R_{kK+1}^{t} R_{kK+2}^{t} \cdots R_{kK+K-1}^{t} E_{(k+1)K}
\end{equation*}
is also a complete residue set modulo \( \widetilde{R}_{k}^{t} \). Note that for each \( k \) and any \( i \in \{1, \cdots, \phi(k)\} \), \((R_{k}^{-1} D_{k}, \frac{1}{m} R_{k}^{t} C_{k,i})\) is a compatible pair. Based on the assumption that $R_{k}^{t}\nu_{k,i_{k}}\in m\mathbb{Z}^{n}$, we can choose these \( i_{k}\in\{1,\cdots,\phi(k)\} \) such that \begin{equation*}
	\widetilde{C}_{k}:=\frac{1}{m}(R_{kK+1}^{t} C_{kK+1,i_{kK+1}}+R_{kK+1}^{t}R_{kK+2}^{t} C_{kK+2,i_{kK+2}}+\dotsi+R_{kK+1}^{t}\dotsi R_{kK+k}^{t} C_{kK+k,i_{kK+k}}).
\end{equation*}
Then $ ({\widetilde{R}_{k}}^{-1}\widetilde{D}_{k},\widetilde{C}_{k}) $ is compatible pair for each $ k\ge0 $ by Lemma \ref{lem2.4} and $ (R_{k}^{-1}D_{k},\frac{1}{m}R_{k}^{t}C_{k,i_{k}}) $ is compatible pair. From $R_{k}^{t}\nu_{k,i_{k}}\in m\mathbb{Z}^{n}$, it follows that $\frac{1}{m}R_{k}^{t}C_{k,i_{k}}\subset E_{k} $, then there exists a set $ L_{k}\subset N_{k} $ with $ 0\in L_{k} $ such that $ \widetilde{C}_{k}=L_{k}\pmod {{\widetilde{R}_{k}}^{t}} $.  In view of Lemma \ref{lem2.4}, we know that \((\widetilde{R}_{k}^{-1} \widetilde{D}_{k}, L_{k})\) is compatible pair for each \(k \ge 0\).

Next, we will show \eqref{3.7}. From \(N_{k} = \widetilde{R}_{k}^{t} (-\frac{1}{2}, \frac{1}{2}]^{n} \cap \mathbb{Z}^{n}\), it follows that
\begin{equation*}
(\widetilde{R}_{k}^{t})^{-1} N_{k} \subset (-\frac{1}{2}, \frac{1}{2}]^{n} \Longrightarrow (\widetilde{R}_{k}^{t})^{-1} L_{k} \subset (-\frac{1}{2}, \frac{1}{2}]^{n} \quad \text{for all } k \ge 0.
\end{equation*}
Combining this with (\ref{3.2}), one has
\begin{equation*}
 \|(\widetilde{R}_{0}^{t}\dotsi \widetilde{R}_{k}^{t})^{-1}L_{0}+\dotsi+(\widetilde{R}_{k-2}^{t}\widetilde{R}_{k-1}^{t}\widetilde{R}_{k}^{t})^{-1}L_{k-2}+(\widetilde{R}_{k-1}^{t}\widetilde{R}_{k}^{t})^{-1}L_{k-1}\|<\frac{c^{2}\sqrt{n}}{2}\frac{r^{K}}{1-r^{K}}.
\end{equation*}
Thus, there exists \( K \ge M \) depending only on \( \delta \) and \( r \) such that
\begin{equation*}
(\widetilde{R}_{0}^{t} \cdots \widetilde{R}_{k}^{t})^{-1} L_{0} + \cdots + (\widetilde{R}_{k-1}^{t} \widetilde{R}_{k}^{t})^{-1} L_{k-1} + (\widetilde{R}_{k}^{t})^{-1} L_{k} \subset [-\frac{1}{2} - \frac{1}{4} \delta, \frac{1}{2} + \frac{1}{4} \delta]^{n}.
\end{equation*}
Hence, we complete the proof of Lemma \ref{lem3.2}.
 \end{proof}
Similar to \eqref{2.2}, we adopt the following notation:
\begin{equation*}
	\mu = \mu_{\{\widetilde{R}_{k}\}, \{\widetilde{D}_{k}\}_{k \ge 0}} \quad \text{and} \quad \mu_{\{\widetilde{R}_{k}\}, \{\widetilde{D}_{k}\}_{l}} = \delta_{\widetilde{R}_{0}^{-1} \widetilde{D}_{0}} * \delta_{\widetilde{R}_{0}^{-1} \widetilde{R}_{1}^{-1} \widetilde{D}_{1}} * \cdots * \delta_{\widetilde{R}_{0}^{-1} \widetilde{R}_{1}^{-1} \cdots \widetilde{R}_{l}^{-1} \widetilde{D}_{l}}.
\end{equation*}
For any $ k\ge1 $, define
\begin{equation}\label{3.8}
 \Lambda_{k}=L_{0}+\widetilde{R}_{0}^{t}L_{1}+\dotsi+\widetilde{R}_{0}^{t}\dotsi\widetilde{R}_{k-1}^{t}L_{k}\quad\text{and}\quad \Lambda=\cup_{k=1}^{\infty}\Lambda_{k}.
\end{equation}
   Clearly, $ \Lambda_{k}\subset\Lambda_{k+1} $, and $ E_{\Lambda} $ is an orthogonal family of $ L^{2}(\mu) $. Our purpose is to prove that $ \Lambda $ is the spectrum of $ \mu $. Recall that
  \begin{equation*}
 Q_{\mu,\Lambda_{k}}(\xi)=\displaystyle{\sum_{\lambda\in \Lambda_{k}}}|\widehat{\mu}(\xi+\lambda)|^2\quad\text{and}\quad Q_{\mu,\Lambda}(\xi)=\displaystyle{\sum_{\lambda\in \Lambda}}|\widehat{\mu}(\xi+\lambda)|^2,\quad \xi\in \mathbb{R}^{n}.
  \end{equation*}
One has $ \lim_{k\rightarrow\infty} Q_{\mu,\Lambda_{k}}(\xi)=Q_{\mu,\Lambda}(\xi).$ By Lemma \ref{lem3.2}, we have
\begin{equation}\label{3.9}
 ({\widetilde{R}_{0}^{t}}\dotsi {\widetilde{R}_{k}^{t}})^{-1}\Lambda_{k}\subset[-\frac{1}{2}-\dfrac{1}{4}\delta,\frac{1}{2}+\dfrac{1}{4}\delta]^{n}.
\end{equation}
 It follows from the fact that \( (\widetilde{R}_{k}^{-1}\widetilde{D}_{k}, L_{k}) \) is a compatible pair and from Lemma \ref{lem2.4} that
\begin{equation*}
 ( (\widetilde{R}_{k} \cdots \widetilde{R}_{0})^{-1} ( \widetilde{D}_{k} + \widetilde{R}_{k} \widetilde{D}_{k-1} + \cdots + \widetilde{R}_{k} \cdots \widetilde{R}_{1} \widetilde{D}_{0}), \Lambda_{k})
\end{equation*}
 is also a compatible pair. Hence, $ \Lambda_{l} $ is a spectrum of $ \mu_{\{\widetilde{R}_{k}\},\{\widetilde{D}_{k}\}_{l}} $. With Lemma \ref{lem2.2}, \begin{equation}\label{3.10}
\displaystyle{\sum_{\lambda\in \Lambda_{l}}}|\widehat{\mu}_{\{\widetilde{R}_{k}\},\{\widetilde{D}_{k}\}_{l}}(\xi+\lambda)|^2=1 \quad \text{for any}\;\xi\in\mathbb{R}^{n}.
 \end{equation}
 One can easily check that
 \begin{equation*}
 \widehat{\mu}_{\{\widetilde{R}_{k}\},\{ \widetilde{D}_{k}\}_{l}}(\xi) = \prod_{i=0}^{l} |m_{\widetilde{D}_{i}}((\widetilde{R}_{0}^{t} \cdots \widetilde{R}_{i}^{t})^{-1} \xi)| = \prod_{i=1}^{(l+1)K} |m_{D_{i}}((R_{1}^{t} \cdots R_{i}^{t})^{-1} \xi)|.
\end{equation*}
Combining Lemma \ref{lem3.1} and Lemma \ref{lem3.2}, we deduce the following lemma.
\begin{lemma}\label{lem3.3}
Let $ \{R_{k}\}_{k=1}^{\infty}\subset M_{n}(\mathbb{Z}) $, $ \mathcal{A}_{\delta,\beta} $ is given by \eqref{a}, and $ K $ is given by Lemma \ref{lem3.2} and satisfies \eqref{3.7}. For any $ \xi\in \mathbb{R}^{n} $, there exists an integer $ N_{\xi}\ge K $, which depends only on $ \xi $ and $ \delta $, such that
\begin{equation*}
 |\widehat{\mu}(\lambda+\xi)|\ge \gamma_{\iota}\gamma^{\iota}|\widehat{\mu}_{\{\widetilde{R}_{k}\},\{\widetilde{D}_{k}\}_{l}}(\lambda+\xi)|
\end{equation*}
for $\iota\ge K $ and $ \lambda\in \Lambda_{k}$, where $ \gamma $ and $ \gamma_{\iota} $ are given by Lemma \ref{lem3.1}.
  \end{lemma}
 \begin{proof}

 It follows from \eqref{3.7} and the fact that \( \widetilde{R}_{0}^{-1} \cdots \widetilde{R}_{k}^{-1} \to \mathbf{0} \) (the zero matrix) as \( k \to \infty \) that there exists \( N_{\xi} \ge K \) such that
\begin{equation}\label{3.11}
 (\widetilde{R}_{0}^{t} \cdots \widetilde{R}_{k}^{t})^{-1}(\lambda + \xi) \subset [-\frac{1}{2} - \delta, \frac{1}{2} + \delta ]^{n}\quad \text{for } k \ge N_{\xi}.
\end{equation}
In addition, since there exists \( l\ge N_{\xi} \) such that \( K(l+1) \ge M \) and \begin{equation*}
 |\widehat{\mu}_{\{\widetilde{R}_{k}\},\{\widetilde{D}_{k}\}_{k>l}}(({\widetilde{R}_{0}^{t}}\dotsi {\widetilde{R}_{l}^{t}})^{-1}(\lambda+\xi))|=|\widehat{\mu}_{k>K(l+1)}(({\widetilde{R}_{0}^{t}}\dotsi {\widetilde{R}_{l}^{t}})^{-1}(\lambda+\xi))|,
\end{equation*} we know from \eqref{3.11} and Lemma \ref{lem3.1} that there exists \( \iota\) such that
\begin{equation*}
 |\widehat{\mu}_{\{\widetilde{R}_{k}\},\{\widetilde{D}_{k}\}_{k>l}}(({\widetilde{R}_{0}^{t}}\dotsi {\widetilde{R}_{l}^{t}})^{-1}(\lambda+\xi))|\ge \gamma_{\iota}\gamma^{\iota}.
\end{equation*}
From $ |\widehat{\mu}(\lambda+\xi)|=|\widehat{\mu}_{\{\widetilde{R}_{k}\},\{\widetilde{D}_{k}\}_{l}}(\lambda+\xi)||\widehat{\mu}_{\{\widetilde{R}_{k}\},\{\widetilde{D}_{k}\}_{k>l}}(({\widetilde{R}_{0}^{t}}\dotsi {\widetilde{R}_{l}^{t}})^{-1}(\lambda+\xi))| ,$ it follows that
 \begin{equation*}
	|\widehat{\mu}(\lambda+\xi)|\ge \gamma_{\iota}\gamma^{\iota}|\widehat{\mu}_{\{\widetilde{R}_{k}\},\{\widetilde{D}_{k}\}_{l}}(\lambda+\xi)|.
\end{equation*}
Hence, the proof is complete.
 \end{proof}
By Lemma \ref{lem2.6}, proving Theorem \ref{thm1.4} only requires demonstrating the following theorem.
 \begin{thm}\label{thm3.4}
Let $ \mu $ be  defined by \eqref{1.1} and \eqref{eq1.2}, satisfying condition \eqref{a}, where $ R_{1}=\text{diag}[m,\cdots,m] $. If there exists $i_{k}\in\{1,\cdots,\phi(k)\}$ such that $R_{k}^{t}\nu_{k,i_{k}}\in m\mathbb{Z}^{n}\;\text{for}\;k\ge 2. $ Then  $ \mu $ is a spectral measure with spectrum $ \Lambda $ given by \eqref{3.8}.
 \end{thm}
 \begin{proof}
 Suppose on the contrary, that $ \Lambda $ defined by \eqref{3.8} is not a spectrum of $ \mu $. Then there exist $ \alpha<1 $ and $ \xi\in\mathbb{R}^{n} $ such that $ Q_{\mu,\Lambda}(\xi)\le\alpha $.
 We choose an increasing integer sequence \( \{n_{k}\}_{k=1}^{\infty} \) such that \( n_{1} = 1 \), and \( n_{k} \) satisfies the following conditions: \( \gamma_{n_{2}} \ge \frac{Q_{\mu,\Lambda}(\xi)}{\alpha} \), \( n_{k+1} - n_{k} \ge n_{2} \ge M \) for \( k \ge 2 \), and
 \begin{equation*}
 ({\widetilde{R}_{0}^{t}} \cdots {\widetilde{R}_{n_{k}}^{t}})^{-1}(\lambda + \xi) \subset [-\frac{1}{2} - \delta, \frac{1}{2} + \delta]^{n} \quad \text{for } k \ge 2 \text{ and } \lambda \in \Lambda_{n_{k}}.
 \end{equation*}
Then,
\begin{equation*}
 |\widehat{\mu}_{j>q+n_{k+1}-n_{k}}(({R_{q+1}^{t}}\dotsi{R_{q+n_{k+1}-n_{k}}^{t}} )^{-1}\xi)|\ge \gamma_{n_{k+1}-n_{k}}\ge \gamma_{n_{2}}\ge\frac{Q_{\Lambda}(\xi)}{\alpha}
\end{equation*}
for $ q\ge0 $, $ k\ge2 $.
Take $ q=n_{k} $, we get
 \begin{align*}
 |\widehat{\mu}_{\{\widetilde{R}_{j}\},\{\widetilde{D}_{j}\}_{j>n_{k+1}}}(({\widetilde{R}_{0}^{t}}\dotsi {\widetilde{R}_{n_{k+1}}^{t}})^{-1}(\xi+\lambda))|
\ge\frac{Q_{\Lambda}(\xi)}{\alpha}.
 \end{align*}
 Hence, for any  $ \lambda\in\Lambda_{n_{k}} \;\text{and}\;k\ge2 $, it follows that
 \begin{align*}
|\widehat{\mu}_{\{\widetilde{R}_{j}\},\{\widetilde{D}_{j}\}_{n_{k+1}}}(\xi+\lambda)|=&|\widehat{\mu}_{\{\widetilde{R}_{j}\},\{\widetilde{D}_{j}\}_{n_{k}}}(\xi+\lambda)||\widehat{\mu}_{\{\widetilde{R}_{j}\},\{\widetilde{D}_{j}\}_{n_{k}< j\le n_{k+1}}}(({\widetilde{R}_{0}^{t}}\dotsi {\widetilde{R}_{n_{k}}^{t}})^{-1}(\xi+\lambda))|\\
 =&|\widehat{\mu}_{\{\widetilde{R}_{j}\},\{\widetilde{D}_{j}\}_{n_{k}}}(\xi+\lambda)|\frac{|\widehat{\mu}_{\{\widetilde{R}_{j}\},\{\widetilde{D}_{j} \}_{j>n_{k}}}(({\widetilde{R}_{0}^{t}}\dotsi {\widetilde{R}_{n_{k}}^{t}})^{-1}(\xi+\lambda))|}{|\widehat{\mu}_{\{\widetilde{R}_{j}\},\{\widetilde{D}_{j}\}_{j>n_{k+1}}}(({\widetilde{R}_{0}^{t}}\dotsi {\widetilde{R}_{n_{k+1}}^{t}})^{-1}(\xi+\lambda))|}\\
 \le&|\widehat{\mu}_{\{\widetilde{R}_{j}\},\{\widetilde{D}_{j}\}_{n_{k}}}(\xi+\lambda)|\frac{|\widehat{\mu}_{\{\widetilde{R}_{j}\},\{\widetilde{D}_{j}\}_{j>n_{k}}}(({\widetilde{R}_{0}^{t}}\dotsi {\widetilde{R}_{n_{k}}^{t}})^{-1}(\xi+\lambda))|\alpha}{Q_{\mu,\Lambda}(\xi)}\\
 =&|\widehat{\mu}(\xi+\lambda)|\frac{\alpha}{Q_{\mu,\Lambda}(\xi)}.\label{3.12}\tag{3.12}
 \end{align*}
 By \eqref{3.12} and Lemma \ref{lem3.3}, we obtain
 \begin{align*}
 Q_{\mu,\Lambda_{n_{k+1}}}(\xi)-Q_{\mu,\Lambda_{n_{k}}}(\xi)=& \displaystyle{\sum_{\lambda\in \Lambda_{n_{k+1}}\backslash\Lambda_{n_{k}}}}|\widehat{\mu}(\xi+\lambda)|^2\\
 \ge&\gamma_{\iota}\gamma^{\iota}\displaystyle{\sum_{\lambda\in \Lambda_{n_{k+1}}\backslash\Lambda_{n_{k}}}}|\widehat{\mu}_{\{\widetilde{R}_{j}\},\{\widetilde{D}_{j}\}_{n_{k+1}}}(\xi+\lambda)|^2\\
 \ge&\gamma_{\iota}\gamma^{\iota}(1-\frac{\alpha}{Q_{\mu,\Lambda}(\xi)}\displaystyle{\sum_{\lambda\in \Lambda_{n_{k}}}}|\widehat{\mu}(\xi+\lambda)|^2)\ge\gamma_{\iota}\gamma^{\iota}(1-\alpha),   \quad k\ge N_{\xi}.
 \end{align*}
 Hence
$$1 > Q_{\mu,\Lambda}(\xi) \ge \sum_{k=N_{\xi}}^{\infty} (Q_{\mu,\Lambda_{n_{k+1}}}(\xi) - Q_{\mu,\Lambda_{n_{k}}}(\xi)) \ge \sum_{k=N_{\xi}}^{\infty} (\gamma_{\iota} \gamma^{\iota} (1 - \alpha)) = +\infty,$$
 which leads to a contradiction. Thus  we complete the proof.
 \end{proof}
\section{\bf  Spectrality of \( \mu_{\{R_{k}\}, \{D_k\}} \) when \( R_{k} \) is a diagonal matrix}
In this section, we aim to prove Theorem \ref{thm1.2}. We consider the Moran measure \( \mu_{\{R_{k}\},\{D_{k}\}} \) generated by the expanding matrix \( R_k = \text{diag}[p_{k,1}, \cdots, p_{k,n}] \) and the digit set \( D_k \) that satisfies condition \eqref{eq1.2}. To demonstrate this,  some technical lemmas are needed.

For convenience, we will continue to use the definition from \eqref{2.2}. Let \( \bm{a} \) be an \( n \)-dimensional vector, and let \( W_{j}(\bm{a}) \) denote the \( j \)-th component of \( \bm{a} \). To simplify the notation, we denote \( j\nu_{k,i} \) in \eqref{eq1.5} as \( \nu_{k,i}^{(j)} \) in this section.
\begin{lemma}[\cite{CLZ}]\label{lem4.3}
	Let \( 0 \in \Lambda \) be a spectrum of \( \mu \), and let \( \Lambda^{'} = \{\lambda_{0} = 0, \lambda_{1}, \cdots, \lambda_{t}\} \subset \Lambda \) be a maximal orthogonal set of \( \mu_{k} \). Define
	\begin{equation}\label{4.11}
		\Lambda^{(0)} = \{\lambda \in \Lambda : \lambda - \lambda_{0} \in  \mathcal{Z}(\widehat{\mu}_{>k_{0}})\cup\{0\} \}
	\end{equation}
	and
\begin{equation}\label{4.12}
	\Lambda^{(i)}=\{\lambda\in\Lambda:\lambda-\lambda_{i}\in \mathcal{Z}(\widehat{\mu}_{>k_{0}})\cup\{0\}\}\backslash(\cup_{j=0}^{i-1}\Lambda^{(j)}),\;i=1,2,\cdots,t.
\end{equation}
	If \( \mathcal{Z}(\widehat{\mu}) \subset R^{-1} \mathbb{Z}^{n} \) for some non-singular integer matrix \( R \), and the set \( \{\Lambda^{(i)} : 0 \leq i \leq t\} \) satisfies the following conditions:

	$\rm(i)$ \( (\Lambda^{(i)} - \Lambda^{(i)}) \setminus \{0\} \subset \mathcal{Z}(\widehat{\mu}_{>k_{0}}) \) for all \( 0 \leq i \leq t \);

	$\rm(ii)$ \( (\Lambda^{(i)} - \Lambda^{(j)}) \subset \mathcal{Z}(\widehat{\mu}_{k_{0}}) \) for \( 0 \leq i \neq j \leq t \).

\noindent Then both \( \mu_{k} \) and \( \mu_{>k} \) are spectral measures.
\end{lemma}

\begin{lemma}\label{lem4.1}
 The measure \( \mu\) has an infinite orthogonal set if and only if there exists an infinite sequence \( \{j_{i_{k}}\}_{k=1}^{\infty} \) such that \( m\mid p_{j_{i_{k},i }} \) for \( 1 \le i \le n \).
\end{lemma}
\begin{proof}
We demonstrate the sufficiency by constructing an infinite orthogonal set. There exists a sequence \( \{l_k\}_{k=1}^{\infty} \) such that \begin{equation*}
l_{k-1}<j_{1_{k_{1}}}\le j_{2_{k_{2}}}\le \cdots \le j_{n_{k_{n}}}<l_{k}.
\end{equation*}
Write
\begin{equation*}
	\lambda_{k}=m^{-1}R_{1}R_{2}\cdots R_{l_{k}}\nu_{l_{k},1}^{(1)}\;\text{and}\;\Lambda=\{\lambda_{k}\}_{k=1}^{\infty}.
\end{equation*}
Thus, for any two distinct $\lambda$, $\lambda^{'}\in\Lambda$, it follows that
\begin{equation*}
	\lambda-\lambda^{'}=\lambda_{i}-\lambda_{j}=m^{-1}R_{1}R_{2}\cdots R_{l_{i}}(\nu_{l_{i},1}^{(1)}-R_{l_{i}+1}\cdots R_{l_{j}}\nu_{l_{j},1}^{(1)}).
\end{equation*}
According to the definition of the sequence \( \{l_{k}\}_{k=1}^{\infty} \), we know that \(R_{l_{i}+1}\cdots R_{l_{j}}\nu_{l_{j},1}^{(1)}\in m\mathbb{Z}^{n} \). Hence,
\begin{equation*}
		\lambda-\lambda^{'}\in R_{1}R_{2}\cdots R_{l_{i}}\mathcal{Z}(\widehat{\delta}_{D_{l_{i}}})\subset\mathcal{Z}(\widehat{\mu}).
\end{equation*}
Consequently, it follows that \( (\Lambda - \Lambda) \setminus \{0\} \subset \mathcal{Z}(\widehat{\mu}) \), which implies that \( \Lambda \) is an infinite orthogonal set for \( \mu \).

Next we prove the necessity. Suppose that \(0\in\Lambda \) is an infinite orthogonal set of \(\mu \). Then
\begin{equation}\label{4.1}
	(\Lambda-\Lambda)\backslash\{0\}\subset\mathcal{Z}(\widehat{\mu})=\cup_{k=1}^{\infty}R_{1}R_{2}\cdots R_{k}\cup_{j=1}^{\phi(k)}\cup_{i=1}^{m-1}(m^{-1}\nu_{k,j}^{(i)}+\mathbb{Z}^{n}).
\end{equation}
According to \eqref{4.1}, it is straightforward to deduce that
\begin{equation}\label{4.2}
	(W_{j}(\Lambda)-W_{j}(\Lambda))\backslash\{0\}\subset\cup_{k=1}^{\infty}p_{1,j}p_{2,j}\cdots p_{k,j}\cup_{l=1}^{m-1}(\dfrac{l}{m}+\mathbb{Z}^{n})
\end{equation}for each \( j\in\{1,2,\cdots,n\} \). In fact, \( W_{j}(\Lambda) \) is an infinite set. Without loss of generality, we only need to prove that $ W_{1}(\Lambda) $ is infinite.
If \( W_{1}(\Lambda) \) is finite, then by the Pigeonhole Principle, there exist two distinct vectors \( \lambda_{1} = (\lambda_{1,1}, \lambda_{2,1}, \cdots, \lambda_{n,1})^{t} \) and \( \lambda_{2} = (\lambda_{1,2}, \lambda_{2,2}, \cdots, \lambda_{n,2})^{t} \in \Lambda \) such that
\begin{equation*}
\lambda_{1}-\lambda_{2}=(0,\lambda_{2,1}-\lambda_{2,2},\cdots,\lambda_{n,1}-\lambda_{n,2})^{t}.
\end{equation*}
Then, $ \lambda_{1}-\lambda_{2}\notin\mathcal{Z}(\widehat{\mu})\cup\{0\} $.
This contradicts the assumption that \( \Lambda \) is an infinite orthogonal set of \( \mu \). Therefore, \( W_j(\Lambda) \) is also an infinite set for all \( j \in \{1, \cdots, n\} \).
In addition, we will prove that for every fixed \(j\in\{1,2,\cdots,n\} \), there are infinite elements in the sequence \(\{p_ {k,j}\}_{k=1}^{\infty} \) that can be divisible by \(m\). Proof by contradiction, there exist \( j\in\{1,2,\cdots,n\} \) and \( N>0 \) such that for \( k > N \), \( m \nmid p_{k,j} \). Then, 
\begin{equation*}
	W_j(\Lambda)\backslash\{0\} \subset	(W_j(\Lambda)-	W_j(\Lambda))\backslash\{0\}\subset \cup_{k=1}^{N}p_{1,j}p_{2,j}\cdots p_{k,j}\cup_{i=1}^{m-1}(\frac{i}{m} +\mathbb{Z}).
\end{equation*}
We obtain \( \# W_j(\Lambda)\le m^{N} \), which contradicts the fact that \(W_j(\Lambda) \) is an infinite set.
Therefore, the proof is complete.
\end{proof}

To prove the necessity of Theorem \ref{thm1.2}, we also require the following key lemma, the proof of which is inspired by \cite{CLZ}.
\begin{lemma}\label{lem4.2}
For $k\ge2, $ if \( \mu_k \) is a spectral measure, then \( m \mid p_{k,j} \), \( 1\le j\le n \).
\end{lemma}
\begin{proof}
Suppose on the contrary that $m\nmid p_{k,j_{0}}$ for some $1\le j_{0}\le n.$ Let $0\in \Lambda$ be a spectrum of $\mu_{k}$, then
\begin{equation*}
	\Lambda\backslash\{0\}\subset (\Lambda-\Lambda)\backslash\{0\}\subset\mathcal{Z}(\widehat{\mu}_{k})= \cup_{i=1}^{k}R_{1}R_{2}\cdots R_{i}\cup_{j=1}^{\phi(i)}\cup_{t=1}^{m-1}(m^{-1}\nu_{i,j}^{(t)}+\mathbb{Z}^{n}).
\end{equation*}
Denote $U_{i}:=R_{1}R_{2}\cdots R_{i}\cup_{j=1}^{\phi(i)}\cup_{t=1}^{m-1}(m^{-1}\nu_{i,j}^{(t)}+\mathbb{Z}^{n})$, then $\mathcal{Z}(\widehat{\mu}_{k})= \cup_{i=1}^{k}U_{i}.$

Next, we prove that \( \#(\Lambda\cap U_{i})<m \) for $i=k-1,k$. If not, we can assume \(\#(\Lambda\cap U_{k-1})=m \), denoted as \( \Lambda\cap U_{k-1}:= \{\lambda_{1}, \lambda_{2},\cdots,\lambda_{m}\} \), where $\lambda_{i}=R_{1}R_{2}\cdots R_{k-1}a_{i}$ and $a_{i}\in\cup_{j=1}^{\phi(k-1)}\cup_{t=1}^{m-1}(m^{-1}\nu_{k-1,j}^{(t)}+\mathbb{Z}^{n}) $ for $i\in\{1,2,\cdots,m\}$. Since for any \( j \in \{1, 2, \cdots, n\} \) and \( i \in \{1, 2, \cdots, m\} \), \( W_j(a_i) \in \cup_{i=1}^{m-1} ( \frac{i}{m} + \mathbb{Z} ) \), by the pigeonhole principle, there must exist two indices \( i_1(j), i_2(j) \in \{1, 2, \cdots, m\} \) such that
\[
W_j(a_{i_1(j)}) - W_j(a_{i_2(j)}) \in \mathbb{Z}.
\]
This also shows that, for any \(j \),
\begin{equation*}
W_j(\lambda_{i_{1}(j)})-W_j(\lambda_{i_{2}(j)})\in p_{1,j}p_{2,j}\cdots p_{k-1,j}\mathbb{Z}.
\end{equation*}
Therefore, \( \lambda_{i_{1}(j)}-\lambda_{i_{2}(j)}\notin R_{1}R_{2}\cdots R_{k-1} \mathcal{Z}(\widehat{\delta}_{D_{k-1}}) \). Letting \( j = j_0 \), and given that \( m \nmid p_{k, j_0} \) and \( m \) is prime, we have \( \lambda_{i_1(j_0)} - \lambda_{i_2(j_0)} \notin R_1 R_2 \cdots R_k \mathcal{Z}(\widehat{\delta}_{D_k}) \).
Clearly, when \( t < k - 1 \), it follows that
\begin{equation*}
\lambda_{i_1(j_0)} - \lambda_{i_2(j_0)} \notin R_1 R_2 \cdots R_t \mathcal{Z}(\widehat{\delta}_{D_t}).
\end{equation*}
Hence, \( \lambda_{i_1(j_0)} - \lambda_{i_2(j_0)} \notin\mathcal{Z}(\widehat{\mu}_{k}) \), which contradicts the fact that \( \Lambda\) is a spectrum of \( \mu_{k} \). Then, we get $\#(\Lambda\cap U_{k-1})<m.$ For \( i = k \), since the proof method is similar, we will not detail it.

It is easy to see that $\Lambda=\{0\}\cup\cup_{i=1}^{k}(\Lambda\cap U_{i}).$ If $k=2$, then
\begin{equation*}
\#\Lambda\le\#(\Lambda\cap U_{1})+\#(\Lambda\cap U_{2})+1\le 2m-1< m^{2},\;\;(m>1).
\end{equation*}
This means that $\Lambda$ is not a spectrum of $\mu_{2}$ since $\dim(L^{2}(\mu_{2}))=m^{2}$, a contradiction. If $k>2$, we set $0\in\Lambda^{'}\subset\Lambda$ is a maximal orthogonal set of $\mu_{k-2}.$ Then, $\#\Lambda^{'}\le m^{k-2}.$ Write $ \Lambda^{'}=\{\lambda_{0},\lambda_{1},\cdots,\lambda_{t}\}$, where $\lambda_{0}=0$, $t\le m^{k-2}-1. $ Then $\Lambda$ can be written as the disjoint union $\Lambda=\cup_{i=0}^{t}\Lambda^{(i)}$, where
\begin{equation*}
	\Lambda^{(0)}=\{\lambda\in\Lambda:\lambda-\lambda_{0}\in U_{k-1}\cup U_{k}\cup \{0\}\}
\end{equation*}
and
\begin{equation*}
\Lambda^{(i)}=\{\lambda\in\Lambda:\lambda-\lambda_{i}\in U_{k-1}\cup U_{k}\cup \{0\}\}\backslash(\cup_{j=0}^{i-1}\Lambda^{(j)}),\;i=1,2,\cdots,t.
\end{equation*}
Clearly, $\Lambda^{(i)}\subset\{\lambda_{i}\}\cup\{\lambda\in\Lambda:\lambda-\lambda_{i}\in U_{k-1}\}\cup\{\lambda\in\Lambda:\lambda-\lambda_{i}\in U_{k}\} .$
It is easy to check that $\#\{\lambda\in\Lambda:\lambda-\lambda_{i}\in U_{j}\}\le m-1 $ for $j=k-1,k.$ In fact, if \(\#\{\lambda \in \Lambda : \lambda - \lambda_{i} \in U_{k-1}\} = m\), then for any \(\alpha, \beta \in \{\lambda \in \Lambda : \lambda - \lambda_{i} \in U_{k-1}\}\),
\begin{equation*}
\alpha - \beta = (\alpha - \lambda_{i}) - (\beta - \lambda_{i}) \in \mathcal{Z}(\widehat{\mu}_{k}).
\end{equation*}
This means that $\{\lambda\in\Lambda:\lambda-\lambda_{i}\in U_{k-1}\}-\lambda_{i}\subset U_{k-1} $ is an orthogonal set of $\mu_{k}$. However, we know that $\#\{\lambda\in\Lambda:\lambda-\lambda_{i}\in U_{k-1}\}\le\#(\Lambda\cap U_{k-1})<m$, a contradicts. Hence, $\#\{\lambda\in\Lambda:\lambda-\lambda_{i}\in U_{k-1}\}\le m-1 $. When \( j = k \), the proof method is similar, and we will omit it. Consequently, we find that
\begin{equation*}
	\#\Lambda\le\sum_{i=0}^{t}\#\Lambda^{(i)}\le m^{k-2}(2m-1)< m^{k}=\dim(L^{2}(\mu_{k})).
\end{equation*}
This implies that \( \Lambda \) is not a spectrum of \( \mu_{k} \), leading to a contradiction. Therefore, we have completed the proof.
\end{proof}
Next, we will prove Theorem \ref{thm1.2}. The sufficiency follows from Theorem \ref{thm1.4}, so we will concentrate on the necessity. For this, we will construct a set $\Lambda^{(i)}$ that satisfies conditions (i) and (ii) of Lemma \ref{lem4.3}, and then use Lemma \ref{lem4.2} to complete the proof.
\begin{proof}[\textbf{ Proof of Theorem \ref{thm1.2} }]
	Suppose that there are $k\ge2$ and $1\le j_{0}\le n$ such that $m\nmid p_{k,j_{0}}$. From Lemma \ref{lem4.1}, it follows that there exists $k_{0}\ge2$ such that $m\nmid p_{k_{0},j_{0}}$ but $m\mid p_{k_{0}+1,j_{0}} .$ Let $0\in\Lambda$ be a spectrum of $\mu$ and let $0\in\Lambda^{'}\subset\Lambda$ is a maximal orthogonal set of $\mu_{k_{0}}.$ Then, $\#\Lambda^{'}\le m^{k_{0}}.$ Write $ \Lambda^{'}=\{\lambda_{0},\lambda_{1},\cdots,\lambda_{t}\}$, where $\lambda_{0}=0$, $t\le m^{k_{0}}-1. $ Similar to \eqref{4.11} and \eqref{4.12}, $\Lambda$ can be expressed as the disjoint union $\Lambda=\cup_{i=0}^{t}\Lambda^{(i)}$, where
	\begin{equation*}
		\Lambda^{(0)}=\{\lambda\in\Lambda:\lambda-\lambda_{0}\in  \mathcal{Z}(\widehat{\mu}_{>k_{0}})\cup\{0\}\}
	\end{equation*}
	and
	\begin{equation*} \Lambda^{(i)}=\{\lambda\in\Lambda:\lambda-\lambda_{i}\in \mathcal{Z}(\widehat{\mu}_{>k_{0}})\cup\{0\}\}\backslash(\cup_{j=0}^{i-1}\Lambda^{(j)}),\;i=1,2,\cdots,t.
	\end{equation*}
Furthermore, we will prove that \( \Lambda^{(i)} \) satisfies the conditions $(i)$ and $(ii)$ of Lemma \ref{lem4.3}, i.e. \\$(i)$ for $0\le i\le t$, \( (\Lambda^{(i)}-\Lambda^{(i)})\backslash\{0\} \subset\mathcal{Z}(\widehat{\mu}_{>k_{0}})\); $(ii)$ for $0\le i\neq j\le t$, \( (\Lambda^{(i)}-\Lambda^{(j)})\subset\mathcal{Z}(\widehat{\mu}_{k_{0}})\).\\
For $(i)$, suppose that there exist two distinct elements $\alpha,\beta\in\Lambda^{(i)}$ such that $\alpha-\beta\in\mathcal{Z}(\widehat{\mu}_{k_{0}})\backslash\mathcal{Z}(\widehat{\mu}_{>k_{0}}) .$ Then
\begin{equation}\label{4.3}
	\alpha-\beta=(\alpha-\lambda_{i})-(\beta-\lambda_{i})\in\mathcal{Z}(\widehat{\mu}_{k_{0}})\backslash\mathcal{Z}(\widehat{\mu}_{>k_{0}}).
\end{equation}
If $\alpha=\lambda_{i} $ or $\beta=\lambda_{i}$, then either $\alpha-\lambda_{i}\in\mathcal{Z}(\widehat{\mu}_{k_{0}})\backslash\mathcal{Z}(\widehat{\mu}_{>k_{0}}) $ or $-(\beta-\lambda_{i})\in\mathcal{Z}(\widehat{\mu}_{k_{0}})\backslash\mathcal{Z}(\widehat{\mu}_{>k_{0}}) $, which leads to a contradiction. If
$\{\alpha-\lambda_{i},\beta-\lambda_{i}\}\subset\mathcal{Z}(\widehat{\mu}_{>k_{0}})$, we may assume \( \alpha-\lambda_{i}=R_{1} R_{2} \cdots R_{t_{1}} a \) and \( \beta-\lambda_{i}=R_{1} R_{2} \cdots R_{t_{2}} b \), where \( t_{1}, t_{2} > k_{0} \) and \( a\in\mathcal{Z}(\widehat{\delta}_{D_{t_{1}}}), b \in \mathcal{Z}(\widehat{\delta}_{D_{t_{2}}}) \). Thus
\begin{equation*}
	W_{j_{0}}(\alpha-\lambda_{i})-W_{j_{0}}(\beta-\lambda_{i})=p_{1,j_{0}}p_{2,j_{0}}\cdots p_{k_{0},j_{0}}(p_{k_{0}+1,j_{0}}\cdots p_{t_{1},j_{0}}W_{j_{0}}(a)-p_{k_{0}+1,j_{0}}\cdots p_{t_{2},j_{0}}W_{j_{0}}(b)).
\end{equation*}
Since $m\mid p_{k_{0}+1,j_{0}}$, then we get
\begin{equation*}
W_{j_{0}}(\alpha-\lambda_{i})-W_{j_{0}}(\beta-\lambda_{i})\in p_{1,j_{0}}p_{2,j_{0}}\cdots p_{k_{0},j_{0}}\mathbb{Z}\nsubseteq W_{j_{0}}(\mathcal{Z}(\widehat{\mu}_{k_{0}})).
\end{equation*}
This implies that \( \alpha-\beta\notin\mathcal{Z}(\widehat{\mu}_{k_{0}})\), which contradicts \eqref{4.3}. Therefore, condition $(i)$ holds.\\
For $(ii)$, suppose that there exist two elements $\alpha\in\Lambda^{(i)},\beta\in\Lambda^{(j)}(i\neq j)$ such that $\alpha-\beta\in\mathcal{Z}(\widehat{\mu}_{>k_{0}})\backslash\mathcal{Z}(\widehat{\mu}_{k_{0}}) .$ If $\alpha=\lambda_{i} $ and $\beta=\lambda_{j}$, then $\alpha-\beta=\lambda_{i}-\lambda_{j}\in(\Lambda^{'}-\Lambda^{'})\backslash\{0\}\subset\mathcal{Z}(\widehat{\mu}_{k_{0}}) ,$ a contradiction. If $\alpha=\lambda_{i} $ or $\beta=\lambda_{j}$, then
 $$\alpha-\beta=(\lambda_{i}-\lambda_{j})-(\beta-\lambda_{j})
\quad {\rm or } \quad \alpha-\beta=(\alpha-\lambda_{i})-(\lambda_{j}-\lambda_{i}). $$
Notice that $\lambda_{i}-\lambda_{j}\in\mathcal{Z}(\widehat{\mu}_{k_{0}}) $ and $\alpha-\lambda_{i},\beta-\lambda_{j}\in\mathcal{Z}(\widehat{\mu}_{>k_{0}}) .$ Using a similar method as for $(i)$, we obtain \( \alpha-\beta\in \mathcal{Z}(\widehat{\mu}_{k_{0}})\), which is a contradiction. If \(\{\alpha - \lambda_{i}, \beta - \lambda_{j}\} \subset \mathcal{Z}(\widehat{\mu}_{>k_{0}})\), then \( (\alpha - \lambda_{i}) - (\beta - \lambda_{j}) \in \mathcal{Z}(\widehat{\mu}_{>k_{0}})\). Similarly,
\begin{equation*}
	\alpha-\beta=(\alpha-\lambda_{i})-(\beta-\lambda_{j})-(\lambda_{j}-\lambda_{i})\in \mathcal{Z}(\widehat{\mu}_{k_{0}}),
\end{equation*}
which is a contradiction. Hence, condition $(ii)$ holds.

Clearly, \( \Lambda \setminus \{0\} \subset \mathcal{Z}(\widehat{\mu}) \subset m^{-1} \mathbb{Z}^{n} \). By Lemma \ref{lem4.3}, we establish that \( \mu_{k_{0}} \) is a spectral measure. However, since \( m \nmid p_{k_{0}, j_{0}} \), this contradicts Lemma \ref{lem4.2}. Thus, we have completed the proof of Theorem \ref{thm1.2}.
\end{proof}

\section{\bf  Proof of Theorem \ref{thm1.3} and some examples}
Firstly, we prove Theorem \ref{thm1.3}. By Lemma \ref{lem2.6}, without loss of generality, we can assume $ R_{1}=\text{diag}[m,\cdots,m] $. Then, from \eqref{2.2}, one has
\begin{equation}\label{5.1}
 \mu=\delta_{\frac{1}{m}D_{1}}*(\mu_{k>1}\circ m) \quad \text{and} \quad \widehat{\mu}(\xi)=m_{D_{1}}(\frac{\xi}{m})\widehat{\mu}_{k>1}(\frac{\xi}{m}).
\end{equation}
When \(\phi(k) = 1\), it follows that
\begin{equation}\label{5.2}
	\mathcal{Z}(\widehat{\mu}) = \cup_{k=1}^{\infty}(R_{k}\cdots R_{1})^{t} \mathcal{Z}(m_{D_{k}}) = \cup_{k=1}^{\infty}(R_{k}\cdots R_{1})^{t} \cup_{l=1}^{m-1} (\dfrac{l\nu_{k}}{m} + \mathbb{Z}^{n}).
\end{equation}
Let $ \Lambda $ be a spectrum of $ \mu $ with $ 0\in \Lambda $. Note that $( \Lambda-\Lambda)\backslash\{0\} \subset \mathcal{Z}(\widehat{\mu})  $. From \eqref{5.2}, it follows that $ \Lambda \subset \mathbb{Z}^{n} $.
Define
\begin{equation*}
 \Phi:=\{(i_{1},\cdots,i_{n})^{t}:i_{1},\cdots,i_{n}\in\{0,1,\cdots,m-1\}\}.
\end{equation*}
For any $ \lambda\in \Lambda, $ there exists a unique $ \kappa \in \Phi $ such that $ \lambda=\kappa+m\omega $ for some $ \omega \in \mathbb{Z}^{n} $. Furthermore, we can define
\begin{equation*}
\Lambda_{\kappa}=\{\omega\in \mathbb{Z}^{n}:\kappa+m\omega\in \Lambda\}
\end{equation*}
such that $ \Lambda $ can be decomposed as
\begin{equation*}
 \Lambda=\cup_{\kappa \in \Phi}(\kappa+m\Lambda_{\kappa}),
\end{equation*}
where $ \kappa+m\Lambda_{\kappa}=\emptyset $ if $ \Lambda_{\kappa}=\emptyset $. One can easily verify that this is a disjoint union. Since $ 0\in \Lambda $, it follows that $  \Lambda_{0}\ne \emptyset.  $

\begin{lemma}\label{lem5.1}
	Let \( \Lambda \) be a spectrum of \( \mu \) with \( 0 \in \Lambda \). Then, for each \( \kappa \in \Phi \), \( \Lambda_{\kappa} \) is either an empty set or an orthogonal set of \( \mu_{k>1} \).
\end{lemma}
\begin{proof}
	Suppose that \( \Lambda_{\kappa} \) is a nonempty set for \( \kappa \in \Phi \), and let \( \omega_{0} \neq \omega_{1} \in \Lambda_{\kappa} \). Then \( \kappa + m\omega_{0} \) and \( \kappa + m\omega_{1} \) belong to \( \Lambda \). By \eqref{5.1}, we obtain
\begin{equation*}
	0 = \widehat{\mu}(m(\omega_{0} - \omega_{1})) = m_{D_{1}}(\omega_{0} - \omega_{1}) \widehat{\mu}_{k>1}(\omega_{0} - \omega_{1}) = \widehat{\mu}_{k>1}(\omega_{0} - \omega_{1}),
\end{equation*}
which implies that \( \Lambda_{\kappa} \) is an orthogonal set of \( \mu_{k>1} \).
\end{proof}
\smallskip
To facilitate the construction of the spectrum of \(  \mu_{k>1} \), we need to define some notations. For each $ k\ge1$, write $\Pi_{k,0}=\{\nu_{k,0}^{(0)}=0,\nu_{k,0}^{(1)},\nu_{k,0}^{(2)},\cdots,\nu_{k,0}^{(m-1)} \},$ where $ \nu_{k,0}^{(i)}=i\nu_{k}\pmod{ m\mathbb{Z}^{n}}\in \Phi $. Define
\begin{equation*}
\Pi_{k,i}=\{\nu_{k,i}^{(0)},\nu_{k,i}^{(1)},\cdots,\nu_{k,i}^{(m-1)}\}=\Pi_{k,0}+ \varphi_{k}(i)\pmod{ m\mathbb{Z}^{n}}\subset\Phi,
\end{equation*}
 where $\varphi_{k}(i)$ is a one-to-one mapping from $ \{1,\cdots,m^{n-1}-1\}$ to $H$, and
 \begin{equation*}
 H:=\{(i_{1},i_{2},\cdots,i_{n-1},0)^{t}\backslash\{0\}:i_{1},\cdots, i_{n-1}\in\{0,1,\cdots,m-1\}\}.
 \end{equation*}
 Note that $ \#(\cup_{i=0}^{m^{n-1}-1}\Pi_{1,i})=\#\Phi  $ and $\cup_{i=0}^{m^{n-1}-1}\Pi_{1,i}\subset \Phi $. This implies that $ \cup_{i=0}^{m^{n-1}-1}\Pi_{1,i}=\Phi. $ Hence,
 \begin{equation}\label{5.5}
 	\Lambda=\cup_{\kappa \in \Phi}(\kappa+m\Lambda_{\kappa})=\cup_{i=0}^{m^{n-1}-1}\cup_{\pi \in \Pi_{1,i}}(\pi+m\Lambda_{\pi}).
 \end{equation}
 Fix $ k=1 $. For each \( i\in\{0,1,\cdots,m^{n-1}-1\}\), we arbitrarily select one element \( \pi_i \) from \( \Pi_{1,i} \). Using the selected \( m^{n-1} \) elements, we can define
 \begin{equation}\label{5.4}
 	\Lambda_{\pi_{0},\cdots,\pi_{ m^{n-1}-1}}:=\cup_{i=0}^{m^{n-1}-1}(\frac{\pi_{i}}{m}+\Lambda_{\pi_{i}}) .
 \end{equation}
Next, we will prove that the set \(\Lambda_{\pi_{0},\cdots,\pi_{ m^{n-1}-1}} \) defined by \eqref{5.4} is a spectrum of \( \mu_{k>1} \).

\begin{lemma}\label{lem5.2}
	Let $ \Lambda $ be a spectrum of $ \mu $ with $ 0\in\Lambda $. Then, for any choice of $ \pi_{i}\in \Pi_{1,i}\;\text{for}\;i=0,1,\cdots,m^{n-1}-1, $ the set $ \Lambda_{\pi_{0},\cdots,\pi_{m^{n-1}-1}} $ is a spectrum of $ \mu_{k>1} $.
\end{lemma}
\begin{proof}
	By Lemma \ref{lem2.2}, equations \eqref{5.1}, \eqref{5.5} and  \eqref{5.4}, it follows that for any $ \xi \in \mathbb{R}^{n} $,
\begin{align*}
		1&=\displaystyle{\sum_{\lambda\in \Lambda}}|\widehat{\mu}(\xi+\lambda)|^2=\sum_{\kappa\in\Phi}\sum_{\lambda\in\Lambda_{\kappa}}|\widehat{\mu}(\xi+\kappa+m\lambda)|^{2}\\
		&=\sum_{\kappa\in\Phi}\sum_{\lambda\in\Lambda_{\kappa}}|\widehat{\delta}_{D_{1}}(\dfrac{\xi+\kappa}{m}+\lambda)|^{2}|\widehat{\mu}_{k>1}(\dfrac{\xi+\kappa}{m}+\lambda)|^{2}\\
		&=\sum_{\kappa\in\Phi}|m_{D_{1}}(\dfrac{\xi+\kappa}{m})|^{2}\sum_{\lambda\in\Lambda_{\kappa}}|\widehat{\mu}_{k>1}(\dfrac{\xi+\kappa}{m}+\lambda)|^{2}\\
		&=\sum_{i=0}^{m^{n-1}-1}\sum_{\pi\in\Pi_{1,i}}|m_{D_{1}}(\dfrac{\xi+\pi}{m})|^{2}\sum_{\lambda\in\Lambda_{\pi}}|\widehat{\mu}_{k>1}(\dfrac{\xi+\pi}{m}+\lambda)|^{2}.
	\end{align*}
For $ \pi \in \Phi ,$ we can express $ \rho_{\pi}(\xi)$ as $|m_{D_{1}}(\frac{\xi+\pi}{m})|^{2} $ and define $ \widetilde{\rho}_{\pi}(\xi)$ as $\sum_{\lambda\in\Lambda_{\pi}}|\widehat{\mu}_{>1}(\frac{\xi+\pi}{m}+\lambda)|^{2}.$
	Then
\begin{equation}\label{5.6}
 \sum_{i=0}^{m^{n-1}-1}\sum_{\pi\in\Pi_{1,i}}\rho_{\pi}\widetilde{\rho}_{\pi}=1.
\end{equation}
Based on the definition of \( \Pi_{1,i} \), we can conclude that for each \( i\in\{0,1,\cdots,m^{n-1}-1\} \), $ (\frac{1}{m}D_{1},\Pi_{1,i}) $ is a compatible pair, which implies $ \sum_{\pi\in\Pi_{1,i}}\rho_{\pi}=1. $

For any $i\ne j\in\{0,1,\cdots,m^{n-1}-1\} $, if we select an arbitrary element \( \pi_{i} \) from \( \Pi_{1,i} \) and an arbitrary element \( \pi_{j} \) from \( \Pi_{1,j} \), then there exist $ \omega_{i}\in\Lambda_{\pi_{i}} $ and $ \omega_{j}\in\Lambda_{\pi_{j}} $ such that \begin{equation*}
0=\widehat{\mu}(\pi_{i}+m\omega_{i}-\pi_{j}-m\omega_{j})=m_{D_{1}}(\frac{\pi_{i}}{m}-\frac{\pi_{j}}{m})  \widehat{\mu}_{k>1}(\frac{\pi_{i}}{m}+\omega_{i}-\frac{\pi_{j}}{m}-\omega_{j}).
\end{equation*}
From the definitions of \( \Pi_{1,i} \) and \( \Pi_{1,j} \), one can verify that \( m_{D_{1}}(\frac{\pi_{i}}{m} - \frac{\pi_{j}}{m}) \neq 0 \). In fact, there exist \( \pi_{i_{0}}, \pi_{j_{0}} \in \Pi_{1,0} \) such that
\begin{equation*}
	\pi_{i} - \pi_{j} = \pi_{i_{0}} - \pi_{j_{0}} + (p_{i_{1}} - p_{j_{1}}, p_{i_{2}} - p_{j_{2}}, \cdots, p_{i_{n-1}} - p_{j_{n-1}}, 0)^{t}.
\end{equation*}
For the set \( \{p_{i_{s}} - p_{j_{s}} : s = 1, 2, \cdots, n-1\} \), there exists at least one \( s \) such that \( m \nmid (p_{i_{s}} - p_{j_{s}}) \). If \( (\pi_{i} - \pi_{j}) \in \Pi_{1,0} \setminus \{0\} + m\mathbb{Z}^{n} \), and since \( (\pi_{i_{0}} - \pi_{j_{0}}) \in \Pi_{1,0} \), it follows that
\begin{equation*}
	(p_{i_{1}} - p_{j_{1}}, p_{i_{2}} - p_{j_{2}}, \cdots, p_{i_{n-1}} - p_{j_{n-1}}, 0)^{t} \in \Pi_{1,0} + m\mathbb{Z}^{n},
\end{equation*}
which leads to a contradiction.
Hence \( \widehat{\mu}_{k>1}(\frac{\pi_{i}}{m} + \omega_{i} - \frac{\pi_{j}}{m} - \omega_{j}) = 0 \). Together with Lemma \ref{lem5.1}, this also demonstrates that   \( \Lambda_{\pi_{0}, \cdots, \pi_{m^{n-1}-1}} \) is an orthogonal set of \( \mu_{k>1} \).
By Lemma \ref{lem2.2}, we get
\begin{equation*}
\sum_{i=0}^{m^{n-1}-1} \sum_{\lambda \in \Lambda_{\pi_{i}}} |\widehat{\mu}_{k>1}(\frac{\pi_{i}}{m} + \xi + \lambda)|^{2} \le 1, \quad \xi \in \mathbb{R}^{n},
\end{equation*}
which implies that
\begin{equation*}
\sum_{i=0}^{m^{n-1}-1} \max_{\pi \in \Pi_{1,i}} \widetilde{\rho}_{\pi} \le 1.
\end{equation*}
Using \eqref{5.6} and Lemma \ref{lem2.1}, for any $ \pi_{i} \in \Pi_{1,i}$, we have $ \widetilde{\rho}_{\pi_{i}}=\max\{\widetilde{\rho}_{\pi}:\pi\in\Pi_{1,i}\}:=\widetilde{\rho}_{i} $ and
\begin{equation*}
 1=\sum_{i=0}^{m^{n-1}-1}\widetilde{\rho}_{i}=\sum_{i=0}^{m^{n-1}-1}\sum_{\lambda\in\Lambda_{\pi_{i}}}|\widehat{\mu}_{k>1}(\dfrac{\xi+\pi_{i}}{m}+\lambda)|^{2}=\sum_{\lambda\in\Lambda_{\pi_{0},\cdots,\pi_{ m^{n-1}-1}}} |\widehat{\mu}_{k>1}(\dfrac{\xi}{m}+\lambda)|^{2}.
\end{equation*}
Consequently, according to Lemma \ref{2.2}, $ \Lambda_{\pi_{0},\cdots,\pi_{ m^{n-1}-1}} $ is a spectrum of \( \mu_{k>1} \).
\end{proof}
\medskip
From the proof of Lemma \ref{lem5.2}, we know that \( \widetilde{\rho}_{0} = \widetilde{\rho}_{\nu_{1,0}^{(1)}} = \widetilde{\rho}_{\nu_{1,0}^{(m-1)}} \) and \( \widetilde{\rho}_{0} > 0 \). Therefore, \( \Lambda_{\nu_{1,0}^{(1)}} \) and \( \Lambda_{\nu_{1,0}^{(m-1)}} \) are non-empty sets. It follows that there exist $ z_{1} $, $ z_{2} \in \mathbb{Z}^{n}$ such that $ \nu_{1}+mz_{1} $ and $ -\nu_{1}+mz_{2} $ lie in $ \Lambda $. In the previous discussion, if we replace \( \mu \) with \( \mu_{\{diag[m,\cdots,m]\}_{k>l}} \), we can derive the following lemma.
\begin{lemma}\label{lem5.3}
	Let $ \Lambda $ be a spectrum of \( \mu_{\{diag[m,\cdots,m]\}_{k>l}} \) with $ 0\in\Lambda $. Then there exist $ z_{1} $ and $ z_{2} \in \mathbb{Z}^{n}$ such that both $ \nu_{l+1}+mz_{1} $ and $ -\nu_{l+1}+mz_{2} $ lie in $ \Lambda $.
\end{lemma}

In the following, we give the proof of Theorem \ref{thm1.3}.
\begin{proof}[\textbf{ Proof of Theorem \ref{thm1.3} }]By Lemma \ref{lem2.6}, without loss of generality, we can assume that $ R_{1}=\text{diag}[m,\cdots,m] $. Let $ 0\in\Lambda $ be a spectrum of $ \mu $, then $ \Lambda\subset \mathbb{Z}^{n} $. For $i\in\{1,2,\cdots,m^{n-1}-1\}$, we can choose $ \pi_{i}=\nu_{1,i}^{(1)} $ and set $\pi_{0}=\nu_{1,0}^{(0)}=0 $. By Lemma \ref{lem5.2}, it follows that $  0\in\Lambda_{\pi_{0},\cdots,\pi_{m^{n-1}-1}} $ is a spectrum of $ \mu_{k>1} $. Then, from Lemma \ref{lem2.6}, we have $ m{R_{2}^{t}}^{-1}\Lambda_{\pi_{0},\cdots,\pi_{m^{n-1}-1}} $ as a spectrum of $ \mu_{\{R_{1}\}_{k>1}} $. According to Lemma \ref{lem5.3}, there exist $ z_{1},z_{2}\in\mathbb{Z}^{n} $ such that
\begin{equation*}
 \dfrac{1}{m}R_{2}^{t}\nu_{2}+R_{2}^{t}z_{1},\;-\dfrac{1}{m}R_{2}^{t}\nu_{2}+R_{2}^{t}z_{2}\in \Lambda_{\pi_{0},\cdots,\pi_{m^{n-1}-1}}=\cup_{i=0}^{m^{n-1}-1}(\frac{1}{m}\pi_{ i}+\Lambda_{\pi_{ i}}).
\end{equation*}
Suppose that $ \frac{1}{m}R_{2}^{t}\nu_{2}\notin\mathbb{Z}^{n} $. Then $\frac{1}{m}R_{2}^{t}\nu_{2}+R_{2}^{t}z_{1},\;-\frac{1}{m}R_{2}^{t}\nu_{2}+R_{2}^{t}z_{2}\notin \Lambda_{\pi_{0}}=\Lambda_{0} $. If there exists $ i\in\{1,2,\cdots,m^{n-1}-1\} $ such that both $ \frac{1}{m}R_{2}^{t}\nu_{2}+R_{2}^{t}z_{1} $ and $ -\frac{1}{m}R_{2}^{t}\nu_{2}+R_{2}^{t}z_{2} $ belong to $ (\frac{1}{m}\pi_{i}+\Lambda_{\pi_{i}}) $, then it follows that $ \pm\frac{1}{m}R_{2}^{t}\nu_{2}-\frac{1}{m}\pi_{ i}\in\mathbb{Z}^{n} $. Consequently, this implies $ -\frac{2}{m}\pi_{ i}=-\frac{2}{m}\nu_{1,i}^{(1)}\in\mathbb{Z}^{n} $, which is impossible. If there exist $ i\neq j\in\{1,2,\cdots,m^{n-1}-1\} $ such that
\begin{equation*}
 \frac{1}{m}R_{2}^{t}\nu_{2}+R_{2}^{t}z_{1}\in\frac{1}{m}\pi_{i}+\Lambda_{\pi_{ i}}  \;\;\text{and}\;\;   -\frac{1}{m}R_{2}^{t}\nu_{2}+R_{2}^{t}z_{2}\in\frac{1}{m}\pi_{j}+\Lambda_{\pi_{j}} .
\end{equation*}
Then,
  \begin{equation*}
 -\frac{1}{m}R_{2}^{t}\nu_{2}-\frac{1}{m}\pi_{j}\in\mathbb{Z}^{n}   \;\;\text{and}\;\;   \frac{1}{m}R_{2}^{t}\nu_{2}-\frac{1}{m}\pi_{i}\in\mathbb{Z}^{n},
\end{equation*}
implies that $ -\frac{1}{m}(\pi_{i}+\pi_{j})=-\frac{1}{m}(\nu_{1,i}^{(1)}+\nu_{1,j}^{(1)})\in\mathbb{Z}^{n} $, which is impossible. Hence, $ \frac{1}{m}R_{2}^{t}\nu_{2}\in\mathbb{Z}^{n} $.
	
Based on the above discussion and the previous lemma, we can replace \( \mu \) with \(\mu_{k>1} \) to obtain $ \frac{1}{m}{R_{3}}^{t}\nu_{3}\in\mathbb{Z}^{n} $. Then, by continuing this iteration, we can establish that $ \frac{1}{m}{R_{k}}^{t}\nu_{k}\in\mathbb{Z}^{n} $ for $k\ge4$. Hence, for $k\ge2$, $ {R_{k}}^{t}\nu_{k}\in m\mathbb{Z}^{n} $. At this point, we have completed the proof of necessity, and sufficiency follows directly from Theorem \ref{thm1.4}.
\end{proof}
Next, we will prove Corollary \ref{cor1.5} and Corollary \ref{cor1.6}.
 \begin{proof}[\textbf{Proof of Corollary \ref{cor1.5}}]
	Without loss of generality, we can assume $R_{1}=diag[m,\cdots,m]$ by Lemma \ref{lem2.6}. Regarding necessity, it follows from Theorem \ref{thm1.3} that \( R_{k}^{t}\nu_{k} \in m\mathbb{Z}^{n} \) for \( k \ge 2 \). Denote $\nu_{k}=(\tilde{\nu}_{k,1},\tilde{\nu}_{k,2},\cdots,\tilde{\nu}_{k,n})^{t} .$ Then
	\begin{equation*}
	R_{k}^{t}\nu_{k}= (\tilde{\nu}_{k,1}a_{1}^{(k)},\sum_{i=1}^{2}\tilde{\nu}_{k,i}a_{i}^{(k)},\cdots,\sum_{i=1}^{n}\tilde{\nu}_{k,i}a_{i}^{(k)})^{t}\in m\mathbb{Z}^{n}.
	\end{equation*}
Since  \( m \) is a prime, we conclude that \( m\mid a_{j}^{(k)} \) for \( k \ge2 \), \( j \in\{1,\cdots,n\} \). Next,  we prove the sufficiency. A simple calculation gives
 \begin{equation*}
 	R_{k}^{t^{-1}}:=\begin{bmatrix}
 		{\frac{1}{a_{1}^{(k)}}} & {0}& {\cdots} &{0}& {0}\\
 		{-\frac{1}{a_{2}^{(k)}}} & {\frac{1}{a_{2}^{(k)}}} & {\cdots} &{0}& {0}\\
 		{\vdots} & {\vdots} &{\ddots} & {\vdots}& {\vdots}\\
 		{0} & {0} & {\cdots} &{\frac{1}{a_{n-1}^{(k)}}} & {0}\\
 		{0} & {0} &{\cdots} &{-\frac{1}{a_{n}^{(k)}}}& {\frac{1}{a_{n}^{(k)}}}\\
 	\end{bmatrix}
 	\quad\text{for}\;k\ge2.
 \end{equation*}
Then  \(\|R_{k}^{t^{-1}}\|^{'} = \max\{\|R_{k}^{t^{-1}}x\| : \|x\| = 1\} < 1\) by

\begin{align*}
	\|R_{k}^{t^{-1}}x\| &= \sqrt{(\frac{x_{1}}{a_{1}^{(k)}})^{2} + (\frac{x_{2}-x_{1}}{a_{2}^{(k)}})^{2} + (\frac{x_{3}-x_{2}}{a_{3}^{(k)}})^{2} + \cdots + (\frac{x_{n}-x_{n-1}}{a_{n}^{(k)}})^{2}} \\
	&\le \frac{1}{m} \sqrt{\sum_{i=1}^{n} x_{i}^{2} + \sum_{i=1}^{n-1} x_{i}^{2} - \sum_{i=1}^{n-1} x_{i} x_{i+1}} \\
	&\le \frac{1}{m} \sqrt{2 + 2\sqrt{1 - x_{n}^{2}}\sqrt{1 - x_{1}^{2}}} \\
	&\le \frac{2}{m} < 1,
\end{align*}
where \(x = (x_{1}, x_{2}, \cdots, x_{n})^{t}\) with \(\|x\| = 1\). Therefore, \( \limsup_{k \to \infty} \| R_{k}^{-1} \|^{'} \leq r < 1 \). Additionally, since \( R_{k} \in M_{n}(m\mathbb{Z}) \), it follows that \( \frac{1}{m} R_{k}^{t}\nu_{k}  \in \mathbb{Z}^{n} \). Furthermore, by choosing \( \delta = \frac{1}{8} \) and \( \beta = \frac{1}{8m} \), we obtain
\begin{equation}\label{5.8}
	R_{k}^{t^{-1}} [-\frac{5}{8}, \frac{5}{8}] \subset [-\frac{5}{8a_{1}^{(k)}}, \frac{5}{8a_{1}^{(k)}}] \times \mathbb{R}^{n-1} \subset [-\frac{5}{8m}, \frac{5}{8m}] \times \mathbb{R}^{n-1}
\end{equation}
for all \( k > 1 \). However, $[-\frac{5}{8m}, \frac{5}{8m}]\cap (\frac{\{1,2,\cdots,m-1\}}{m}+\mathbb{Z})_{\frac{1}{8m}}=\emptyset $. Together with \eqref{5.8}, this means that for any \( k > 1 \) and \( i \geq 0 \),
\begin{equation*}
R_{k}^{t} R_{k+1}^{t} \cdots R_{k+i}^{t} \in \mathcal{A}_{\frac{1}{8}, \frac{1}{8m}}.
\end{equation*}
Therefore, based on Theorem \ref{thm1.3}, we can establish sufficiency.
\end{proof}
 \begin{proof}[\textbf{Proof of Corollary \ref{cor1.6}}]
According to Theorem \ref{thm1.3}, we only need to prove that \( \mathcal{Z}(m_{D_{k}}) \) satisfies \eqref{eq1.2}. For $k\ge1$, $  \mathcal{Z}(m_{D_{k}}) =0$ if and only if
\begin{equation*}
\begin{cases}
		x_{1}=\dfrac{d_{k}-2b_{k}+k_{1}}{3(a_{k}d_{k}-c_{k}b_{k})}\\
		x_{2}=\dfrac{c_{k}-2a_{k}+k_{2}}{3(c_{k}b_{k}-a_{k}d_{k})}
\end{cases}\;\;\text{or}\;\;
\begin{cases}
	x_{1}=\dfrac{2d_{k}-b_{k}+k_{1}^{'}}{3(a_{k}d_{k}-c_{k}b_{k})}\\
	x_{2}=\dfrac{2c_{k}-a_{k}+k_{2}^{'}}{3(c_{k}b_{k}-a_{k}d_{k})}
\end{cases}
\;\;\text{for}\;\; k_{1}^{'},k_{2}^{'},k_{2},k_{1}\in\mathbb{Z}.
\end{equation*}
Since $|a_{k}d_{k}-c_{k}b_{k}|=1,$  then
\begin{equation*}
	\begin{cases}
		x_{1}=\pm\dfrac{d_{k}-2b_{k}+k_{1}}{3}\\
		x_{2}=\mp\dfrac{c_{k}-2a_{k}+k_{2}}{3}
	\end{cases}\;\;\text{or}\;\;
	\begin{cases}
		x_{1}=\pm\dfrac{2d_{k}-b_{k}+k_{1}^{'}}{3}\\
		x_{2}=\mp\dfrac{2c_{k}-a_{k}+k_{2}^{'}}{3}
	\end{cases}
	\;\;\text{for}\;\; k_{1}^{'},k_{2}^{'},k_{2},k_{1}\in\mathbb{Z}.
\end{equation*}
Additionally, from the assumed conditions, we get
\begin{equation*}
	\begin{cases}
	(2b_{k}-d_{k})=(c_{k}-2a_{k})\pmod3,\;(2c_{k}-a_{k})=(b_{k}-2d_{k})\pmod3 &\quad\text{if}\;D_{k}\in\Gamma_{1}\\
(d_{k}-2b_{k})=(c_{k}-2a_{k})\pmod3,\;(2c_{k}-a_{k})=(2d_{k}-b_{k})\pmod3 &\quad\text{if}\;D_{k}\in\Gamma_{2}
	\end{cases}
\end{equation*}
and $\{1,2\}\equiv\{(2b_{k}-d_{k}),(2c_{k}-a_{k})\}\equiv\{(d_{k}-2b_{k}),(2c_{k}-a_{k})\}\pmod3 $. This means that if \( D_{k}\in\Gamma_{1} \), then \( \mathcal{Z}(m_{D_{k}})=\{\frac{1}{3}(1,1)^{t}+\mathbb{Z}^{2}\}\cup\{\frac{2}{3}(1,1)^{t}+\mathbb{Z}^{2}\} \); if \( D_{k}\in\Gamma_{2} \), then \( \mathcal{Z}(m_{D_{k}}) =\{\frac{1}{3}(1,2)^{t}+\mathbb{Z}^{2}\}\cup\{\frac{2}{3}(1,2)^{t}+\mathbb{Z}^{2}\} \). Therefore, \eqref{eq1.2} is satisfied, and the proof is complete.
\end{proof}
In the final part of this section, we provide several examples to illustrate our theorem. First, we reference two examples from \cite{LLZ} to demonstrate that the condition \eqref{1.2} in Theorem \ref{thm1.4} and Theorem \ref{thm1.2} is non-removable.
\begin{exam}\label{exam1}
Let $R=\text{diag}[9,9]$ and $D_{k}=\{(0,0)^{t},(1,0)^{t},(2^{2k},1)^{t}\}$ for all $k\ge1$. Then the associated Moran-Sierpinski type measure $\mu_{\{R\}, \{D_{k}\}}$ with compact support is not a spectral measure.
\end{exam}
We can easily calculate that for each \( k \ge 1 \), \( \mathcal{Z}(m_{D_{k}}) = \pm\frac{1}{3}(1,1)^{t} + \mathbb{Z}^{2} \). Clearly, this example satisfies condition \eqref{a}, and \( \frac{1}{3} R^{t}\nu_{k}  \subset \mathbb{Z}^{2} \), where \( \nu_{k} = (1,1)^{t} \). It is also straightforward to see that \( \sup\{\|d\| : d \in D_{k}, k \ge 1\} = \infty \) in this example. However, using a proof method similar to that in \cite[Proposition 5.1]{LLZ}, we can conclude that \( \mu_{\{R\}, \{D_{k}\}} \) is not a spectral measure. The following example implies that \( \limsup_{k\rightarrow\infty} \| R_{k}^{-1}\|^{'} \leq r < 1 \) is non-removable.
\begin{exam}\label{exam2}
Let $p\in3\mathbb{Z}$, and integers $\{a_{k}\}_{k=1}^{\infty}$ be a sequence with $\sum_{i=1}^{k}a_{k}=-2^{k}p$ for all $k\ge1.$ Suppose that
\begin{equation*}
	R_{k}=\begin{bmatrix}
		p & a_{k} \\
		0     & p
	\end{bmatrix}, \; D=\{(0,0)^{t}, (1,0)^{t}, (0,1)^{t}\}
\end{equation*}
for all $k\ge1$. Then the associated Moran-Sierpinski type measure $\mu_{\{R_{k}\}, \{D\}}$ with compact support is not a spectral measure.
\end{exam}

\begin{exam}\label{exam3}
	Let $ R_{k}=\text{diag}[p_{k,1},p_{k,2}]\subset M_{2}(\mathbb{Z})$ and $D_{k}\in\{B_{1},B_{2},B_{3}\}$ for all $k\ge1$, where
\begin{align*}
	\begin{cases}
		B_{1}=\{(0,0)^{t}, (1,0)^{t}, (0,1)^{t}, (1,1)^{t}, (3,3)^{t}\}\\
		B_{2}=\{(0,0)^{t}, (1,0)^{t}, (0,-1)^{t}, (1,-1)^{t}, (3,-3)^{t}\}\\
		B_{3}=\{(0,0)^{t}, (1,0)^{t}, (1,1)^{t}, (2,1)^{t}, (2,2)^{t}\}.
	\end{cases}
\end{align*}
Then $ \mu_{\{R_{k}\},\{D_{k}\}} $ is a spectral measure if and only if \( 5 \mid p_{k,i} \) for \( i = 1, 2\) and \( k \ge 2 \).
\end{exam}
\begin{proof}
	Let $\xi=(\xi_{1},\xi_{2})^{t}\in[0,1)^{2}$, then $\mathcal{Z}(m_{B_{1}})=0$  implies that
	\begin{equation*}
		\cos\pi(\xi_{1}-\xi_{2})+\cos\pi(\xi_{1}+\xi_{2})=-\frac{1}{2}e^{5\pi i (\xi_{1}+\xi_{2})}.
	\end{equation*}
	In that case, there is definitely $ 5(\xi_{1}+\xi_{2})\in\{0,1,2,\cdots,9\} .$ If $  5(\xi_{1}+\xi_{2})\in\{0,1,2,5,8,9\}$, then it has no solution in $[0, 1)^2$.
	When $  5(\xi_{1}+\xi_{2})=3, 4, 6, 7$, we get
	\begin{center}
		\begin{tabular}{|Sc|Sc|Sc|Sc|Sc|}
			\hline
			$\xi_1+\xi_2$ & $\xi_{1}+\xi_{2}=\frac{3}{5}$     &$\xi_{1}+\xi_{2}=\frac{4}{5}$     & $\xi_{1}+\xi_{2}=\frac{6}{5}$     & $\xi_{1}+\xi_{2}=\frac{7}{5}$     \\  \hline
			$(\xi_1,\xi_2)$ &$(\frac{1}{5}, \frac{2}{5})^t, (\frac{2}{5}, \frac{1}{5})^t$ & $(\frac{1}{5}, \frac{3}{5})^t, (\frac{3}{5}, \frac{1}{5})^t$ & $(\frac{2}{5}, \frac{4}{5})^t, (\frac{4}{5}, \frac{2}{5})^t$ & $(\frac{3}{5}, \frac{4}{5})^t, (\frac{4}{5}, \frac{3}{5})^t$ \\
			\hline
		\end{tabular}
	\end{center}
	Hence,
\begin{equation*}
	\mathcal{Z}(m_{B_{1}})=(\cup_{j=1}^{4}(\frac{j}{5}(1,2)^{t}+\mathbb{Z}^{2}))\cup(\cup_{j=1}^{4}(\frac{j}{5}(1,3)^{t}+\mathbb{Z}^{2})).
\end{equation*}
For the digit set $B_{2}$, 	let $\xi=(\xi_{1},\xi_{2})^{t}\in[0,1)^{2}$. Then, if $\mathcal{Z}(m_{B_{1}})=0$, it  implies that
\begin{equation*}
	\cos\pi(\xi_{1}+\xi_{2})+\cos\pi(\xi_{2}-\xi_{1})=-\frac{1}{2}e^{5\pi i (\xi_{2}-\xi_{1})}.
\end{equation*}
In this case, we have \( 5(\xi_{2} - \xi_{1}) \in \{0, \pm 1, \pm 2, \pm 3, \pm 4\} \). If \( 5(\xi_{2} - \xi_{1}) \in \{0, \pm 3, \pm 4\} \), then there are no solutions in \( [0, 1)^{2} \). When \( 5(\xi_{2} - \xi_{1}) = \pm 1 \) or \( \pm 2 \), we obtain the following results:
\begin{center}
	\begin{tabular}{|Sc|Sc|Sc|Sc|Sc|}
		\hline
		$\xi_2-\xi_1$ & $\xi_{2}-\xi_{1}=-\frac{1}{5}$     &$\xi_{2}-\xi_{1}=\frac{1}{5}$     & $\xi_{2}-\xi_{1}=-\frac{2}{5}$     & $\xi_{2}-\xi_{1}=\frac{2}{5}$     \\  \hline
		$(\xi_1,\xi_2)$ &$(\frac{2}{5}, \frac{1}{5})^t, (\frac{4}{5}, \frac{3}{5})^t$ & $(\frac{1}{5}, \frac{2}{5})^t, (\frac{3}{5}, \frac{4}{5})^t$ & $(\frac{3}{5}, \frac{1}{5})^t, (\frac{4}{5}, \frac{2}{5})^t$ & $(\frac{1}{5}, \frac{3}{5})^t, (\frac{2}{5}, \frac{4}{5})^t$ \\
		\hline
	\end{tabular}
\end{center}
Therefore, we obtain that
\begin{equation*}
	\mathcal{Z}(m_{B_{2}})=((1,
			2)^{t}+\mathbb{Z}^{2}))\cup(\cup_{j=1}^{4}(\frac{j}{5}(1,
			3)^{t}+\mathbb{Z}^{2})).
\end{equation*}
Regarding $m_{B_{3}}(\xi)=0$, a simple calculation shows that:
 \begin{equation*}
	-1=(1+e^{-2\pi i(\xi_{1}+\xi_{2})})(e^{-2\pi i\xi_{1}}+e^{-2\pi i(\xi_{1}+\xi_{2})}).
\end{equation*}
By further examining the argument and modulus on both sides of the equation, we can deduce that \( \cos(\pi \xi_{2}) \cos(\pi (\xi_{1} + \xi_{2})) = \frac{1}{4} \) and \( 3\xi_{1} + 2\xi_{2} \in 2\mathbb{Z} + 1 \). Thus,
\begin{equation*}
	\mathcal{Z}(m_{B_{3}})=\cup_{j=1}^{4}(\frac{j}{5}(1,
	1)^{t}+\mathbb{Z}^{2}).
\end{equation*}
Based on the above calculation, it is found that \( B_{1} \), \( B_{2} \) and \( B_{3} \) all satisfy the equation \eqref{eq1.2}. The proof is then completed directly using Theorem \ref{thm1.2}.
\end{proof}
Let $R_{1}=\text{diag}[5,5] $ and $R_{k}=\text{diag}[10,5] $ for all $k\ge2,$ \[ B=\{(0,0)^{t}, (1,0)^{t}, (1,1)^{t}, (2,1)^{t}, (2,2)^{t}\}.\]
From \ref{exam3}, \( \mu_{\{R_{k}\}, \{B\}} \) is a spectral measure. Note that \( T(\{R_{k}\}, B) \subset [0, 1]^2 \), as shown in Figs. \ref{Fractal-1}, \ref{Fractal-2} and \ref{Fractal-3}, which also represent the first three approximations of the Sierpinski fractal.

\begin{figure}[ht]
	\centering
	\begin{minipage}{0.30\textwidth}
		\centering
		\includegraphics[width=\textwidth]{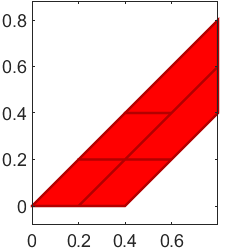}
		\caption{}
		\label{Fractal-1}
	\end{minipage}
	\hfill 
	\begin{minipage}{0.30\textwidth}
		\centering
		\includegraphics[width=\textwidth]{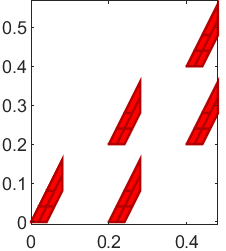}
		\caption{}
		\label{Fractal-2}
	\end{minipage}
	\hfill 
	\begin{minipage}{0.30\textwidth}
		\centering
		\includegraphics[width=\textwidth]{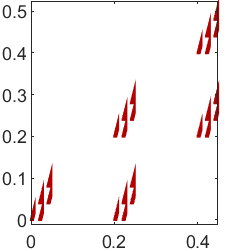}
		\caption{}
		\label{Fractal-3}
	\end{minipage}
\end{figure}
\medspace
\begin{exam}\label{exam4}
	Let $ 	R_{k}=\begin{bmatrix}
	a_{k} & a_{k} \\
	0     & b_{k}
\end{bmatrix}\subset M_{2}(\mathbb{Z})$ and $D_{k}\in\{B_{1},B_{2}\}$ for all $k\ge1$, where
\begin{equation*}
	B_{1}=\{(0,0)^{t}, (1,2)^{t}, (1,3)^{t}\}
	,\;  B_{2}=\{(0,0)^{t}, (2,3)^{t}, (3,5)^{t}\}.
\end{equation*}
Then $ \mu_{\{R_{k}\},\{D_{k}\}} $ is a spectral measure if and only if \(R_{k}\in M_{2}(3\mathbb{Z})\) for \( k \ge 2 \).
\end{exam}
\begin{proof}
	Direct verification of Corollary \ref{cor1.5} is sufficient to prove the result; therefore, the details are omitted.
\end{proof}
Finally, a natural question arises as follows.
\begin{que}
When \( \phi(k) \neq 1 \) in \eqref{eq1.2}, is the conclusion of Theorem \ref{thm1.4} both necessary and sufficient?
\end{que}

\section*{Acknowledgements}
\thanks{The authors thank Ming Liang Chen for valuable suggestions on this work. The research is supported in part by the NSFC (No.12371072).}

\textbf{Competing interests.} The authors declare that they have no competing interests.

\textbf{Data availability.}
Data availability is not applicable to this article as no new data were created or analyzed in this study.


\begin{thebibliography}{99}
\bibitem{AFL1} L.X. An, X.Y. Fu, C.K. Lai, On spectral Cantor-Moran measures and a variant of Bourgain's sum of sine problem. Adv. Math. 349 (2019) 84-124.
\bibitem{CYZ} J.L. Chen, X.Y. Yan, P.F. Zhang, The cardinality of orthogonal exponentials for a class of self-affine measures on $\mathbb{R}^n$. Acta Math. Hungar. 175 (2025), 219-235.
\bibitem{CCWW} M.L. Chen, J. Cao, J.L. Wang, Y. Wang, On the spectra of a class of Moran measures, Forum Math. 36 (2024)  377-387.
\bibitem{CLZ} Z.C. Chi, J.F. Lu, M.M. Zhang, A class of spectral Moran measures generated by the compatible tower. J. Geom. Anal. 34 (2024) 201.
\bibitem{C} J.L. Chen, The spectral study of a class of Moran measures in $\mathbb{R}^n$. J. Math. Anal. Appl. 548 (2025), 129384.
\bibitem{DHL} X.R. Dai, X.G. He, C.K. Lai,
Spectral property of Cantor measures with consecutive digits. Adv. Math. 242 (2013) 187-208.
\bibitem{DHL1} X.R. Dai, X.G. He, K.S. Lau, On spectral N-Bernoulli measures. Adv. Math. 259 (2014) 511-531.
\bibitem{DHLY} Q.R. Deng, X.G. He, M.T. Li, Y.L. Ye, The orthogonal bases of exponential functions based on Moran-Sierpinski measures. Acta Math. Sin. (Engl. Ser.) 40 (2024) 1804-1824.
\bibitem{DL1} Q.R. Deng, M.T. Li, Spectrality of Moran-type self-similar measures on $\mathbb{R}$. J. Math. Anal. Appl. 506 (2022) 125547.
\bibitem{DL}Q.R. Deng, K.S. Lau, Sierpinski-type spectral self-similar measures, J. Funct. Anal. 269 (2015) 1310-1326.
\bibitem{DH1} D.E. Dutkay, J. Haussermann, C.K. Lai, Hadamard triples generate self-affine spectral measures. Trans. Amer. Math. Soc. 371 (2019) 1439-1481.
\bibitem{B1} B. Fuglede,
Commuting self-adjoint partial differential operators and a group theoretic problem. J. Funct. Anal. 16 (1974) 101-121.
\bibitem{F1} K.J. Falconer, Fractal geometry: mathematical foundations and applications. John Wiley  Sons, 2004.
\bibitem{FHW} Y.S. Fu, X.G. He, Z.X. Wen,
Spectra of Bernoulli convolutions and random convolutions. J. Math. Pures Appl. 116 (2018) 105-131.
\bibitem{HLL1}X.G. He, C.K. Lai, K.S. Lau, Exponential spectra in $ L^{2}(\mu) $. Appl. Comput. Harmon. Anal. 34 (2013) 327-338.
\bibitem{HH1} L. He, X.G. He, On the Fourier orthonormal bases of Cantor-Moran measures. J. Funct. Anal. 272 (2017) 1980-2004.
\bibitem{JP1}P.E.T. Jorgensen, S. Pedersen, Dense analytic subspaces in fractal $ L^{2} $-spaces. J. Anal. Math. 75 (1998) 185-228.
\bibitem{J}J. Kigami, Analysis on Fractals, Cambridge Tracts in Mathematics, vol. 143, Cambridge University
Press, Cambridge, 2001, viii+226 pp. ISBN: 0-521-79321-1.
\bibitem{KM1}  M. Kolountzakis, M. Matolcsi,  Tiles with no spectra. Forum Math. 18 (2006) 519-528.
\bibitem{LDL} J.C. Liu, X.H. Dong, J.L. Li, Non-spectral problem for the planar self-affine measures, J. Funct. Anal. 273 (2017) 705-720.
\bibitem{LLZ} J.S. Liu, Z.Y. Lu, T. Zhou, Spectrality of Moran-Sierpinski type measures. J. Funct. Anal. 284 (2023) 109820.
\bibitem{LDZ}Z.Y. Lu, X.H. Dong, P.F. Zhang, Spectrality of some one-dimensional Moran measures, J. Fourier
Anal. Appl. 28 (2022) 1-22.
\bibitem{LW1} I. Laba, Y. Wang, On spectral Cantor measures. J. Funct. Anal. 193 (2002) 409-420.
\bibitem{LM}N. Lev, M. Matolcsi. The Fuglede conjecture for convex domains is true in all dimensions. Acta Math. 228(2) (2022) 385-420.
\bibitem{S1} R. Strichartz, Mock Fourier series and transforms associated with certain Cantor measures. J. Anal. Math. 81 (2000) 209-238.
\bibitem{S} R.S. Strichartz, Remarks on: "Dense analytic subspaces in fractal $ L^{2} $-spaces". J. Anal. Math. 75 (1998) 229-231.
\bibitem{W} W. Sierpinski, General Topology. Mathematical Expositions, No. 7, University of Toronto Press,
Toronto, 1952, Translated by C. Cecilia Krieger.
\bibitem{TY1}M.W. Tang, F.L. Yin, Spectrality of Moran measures with four-element digit sets.  J. Math. Anal. Appl. 461 (2018) 354-363.
\bibitem{T1}T. Tao, Fuglede's conjecture is false in $5$ and higher dimensions. Math. Res. Lett. 11 (2004) 251-258.
\bibitem{WLS1} Z.Y. Wang, J.C. Liu, J. Su,  Spectral property of self-affine measures on $\mathbb{R}^{n}$. J. Fourier Anal. Appl. 27 (2021) 79.
\bibitem{Y1}  Z.H. Yan, Spectral Moran measures on $ \mathbb{R}^{2} $. Nonlinearity 35 (2022) 1261-1285.
\end{thebibliography}
\end{document}